\documentclass[a4paper,11pt,oneside,reqno]{amsart}
\usepackage[utf8]{inputenc}
\usepackage{amsmath,amsthm,amsfonts,latexsym,amssymb,bm,enumerate}
\usepackage{ae}
\usepackage{cite}
\usepackage{float}
\usepackage{lmodern}
\usepackage[T1]{fontenc}

\usepackage{color}

\usepackage[colorlinks=true]{hyperref}

\definecolor{dark-red}{rgb}{.54,.0,.0}
\definecolor{dark-green}{rgb}{.0,.4,.0}
\definecolor{dark-blue}{rgb}{.04,.04,.4}

\hypersetup{linkcolor=dark-red, urlcolor=dark-blue, citecolor=dark-green}

\newcommand{\nc}{\newcommand}
\newcommand{\Rb}{{\mathbb{R}}}
\renewcommand{\phi}{\varphi}
\nc{\txt}{\textstyle}
\nc{\be}{\begin{equation}}
\nc{\ee}{\end{equation}}
\nc{\ba}{\begin{eqnarray}}
\nc{\ea}{\end{eqnarray}}
\nc{\bas}{\begin{eqnarray*}}
\nc{\eas}{\end{eqnarray*}}
\nc{\weak}{\rightharpoonup}
\nc{\paq}{$({\cal{P}}_{\alpha,q})$}
\nc{\paz}{$({\cal{P}}_{\alpha,\tz})$}
\nc{\pazero}{$({\cal{P}}_{0,\tz})$}

\nc{\Om}{\Omega}
\nc{\ek}{{\eps_{k}}}
\nc{\ak}{{\alpha_{k}}}
\nc{\ck}{{C_{k}}}
\nc{\ct}{\tilde{C}}
\nc{\cc}{\hat{C}}
\nc{\uk}{{u_{k}}}
\nc{\Uk}{{U_{k}}}
\nc{\ts}{{2^*}}
\nc{\tsu}{{{2^*}-1}}
\nc{\tsd}{{{2^*}-2}}
\nc{\tst}{{{2^*}-3}}
\nc{\tz}{{2^\#}}
\nc{\tzu}{{{2^\#}-1}}
\nc{\tzd}{{{2^\#}-2}}
\nc{\wk}{{w_{k}}}
\nc{\sndd}{\frac{S^\frac N2}2}
\nc{\sddn}{\frac{S}{2^{\frac 2N}}}

\nc{\igukw}{{\int\nabla\Uk\cdot\nabla\wk}}
\nc{\iukw}{{\int U_k\wk}}
\nc{\iuks}{{|U_k|_{2^*}^{2^*}}}
\nc{\iukz}{{|U_k|_{2^\#}^{2^\#}}}
\nc{\iukpz}{{|u_k|_{2^\#}^{2^\#}}}
\nc{\iusu}{{\int\ U_k^{{2^*}-1} w_k}}
\nc{\iusd}{{\int\ U_k^{{2^*}-2} w_k^2}}
\nc{\iust}{{\int\ U_k^{{2^*}-3} w_k^3}}
\nc{\iuzu}{{\int\ U_k^{{2^\#}-1} |w_k|}}
\nc{\iuzd}{{\int\ U_k^{{2^\#}-2} w_k^2}}
\nc{\iuz}{{|u|_{\tz}^\tz}}
\nc{\ius}{{|u|_{\ts}^\ts}}

\nc{\nguk}{{|\nabla\Uk|_{2}}}
\nc{\ngukq}{{|\nabla\Uk|_{2}^2}}
\nc{\ngwkq}{{|\nabla\wk|_{2}^2}}
\nc{\nuk}{{|\Uk|_{2}}}
\nc{\nukq}{{|\Uk|_{2}^2}}
\nc{\nuks}{{|\Uk|_{\ts}}}
\nc{\nuksq}{{|\Uk|_{\ts}^2}}
\nc{\nukss}{{|\Uk|_{\ts}^\ts}}
\nc{\nuksmd}{{|\Uk|_{\ts}^{-2}}}
\nc{\nwk}{{||\wk||}}
\nc{\nwks}{{|\wk|_{\ts}}}
\nc{\nwksq}{{|\wk|_{\ts}^2}}
\nc{\ndwkq}{{|\wk|_{2}^2}}
\nc{\nwkq}{{||\wk||^2}}
\nc{\nwkr}{{||\wk||^r}}
\nc{\ngukpq}{{|\nabla u_{k}|_{2}^2}}
\nc{\ngupq}{{|\nabla u|_{2}^2}}
\nc{\nukps}{{|u_{k}|_{\ts}}}
\nc{\nukpss}{{|u_{k}|_{\ts}^\ts}}
\nc{\nups}{{|u|_{\ts}}}
\nc{\nupss}{{|u|_{\ts}^\ts}} 
\nc{\nupsz}{{|u|_{\ts}^\tz}} 
\nc{\nukpsq}{{|u_{k}|_{\ts}^2}}
\nc{\nupsq}{{|u|_{\ts}^2}}
\nc{\nupq}{{|u|_{2}^2}}
\nc{\nukpq}{{|u_k|_{2}^2}}

\nc{\ql}{\frac{\ngukq}{|\Uk|_{\ts}^{2^*}}}
\nc{\beuk}{\frac{\ngukq+a\nukq}{\nuksq}}

\nc{\nhu}{||u||^2}
\nc{\pnhu}{||u||^2}
\nc{\pnhumeio}{||u||}
\nc{\nhukpq}{||u_{k}||^2}
\nc{\pnhukpq}{||u_{k}||^2}
\nc{\pnhukpqmeio}{||u_{k}||}
\nc{\nhukq}{||U_{k}||^2}
\nc{\pnhukq}{||U_{k}||^2}
\nc{\nhwkq}{||w_{k}||^2}
\nc{\pnhwkq}{||w_{k}||^2}

\nc{\ttt}{{\Rb^{N}_{+}}}
\nc{\ue}{U_{\eps}}
\nc{\pb}{\bar{\phi}_{\eps}}
\nc{\pbq}{\bar{\phi}_{\eps}^2}
\nc{\pk}{{\phi}_{\eps}}
\nc{\uz}{U^\tzd_{\eps}}
\nc{\uzz}{U^\tzd_{\eps,y_{\eps}}}
\nc{\us}{U^\tsd_{\eps}}
\nc{\uss}{U^\tsd_{\eps,y_{\eps}}}
\nc{\et}{\eta_{\eps}}
\nc{\etq}{\eta_{\eps}^2}
\nc{\pt}{{\tilde\phi}}
\nc{\pte}{{\tilde{\phi}_{\eps}}}
\nc{\ptq}{{{\tilde\phi}^2}}
\nc{\pteq}{\tilde{\phi}_{\eps}^2}
\nc{\ome}{{\Om_{\eps}}}

\nc{\mk}{M_{k}}
\nc{\dk}{\delta_{k}}
\nc{\pkk}{P_{k}}
\nc{\ndd}{{\frac{N-2}{2}}}
\nc{\ium}{U_{\delta_{k},P_{k}}}
\nc{\vk}{{v_{k}}}

\nc{\gb}{\gamma+\sqrt{\gamma^2+4\beta}}
\nc{\gbb}{\gamma_{\infty}+\sqrt{\gamma_{\infty}^2+4\beta_{\infty}}}
\nc{\ia}{\frac{1}{4(\tz)^{\frac{2}{N}}}
         \left[\left(\gb\right)^{N}+2\cdot\ts\beta\left(\gb\right)^{N-2}
         \right]^{\frac{2}{N}}}
\nc{\iab}{\frac{1}{4(\tz)^{\frac{2}{N}}}
         \left[\left(\gbb\right)^{N}+2\cdot\ts\beta_{\infty}\left(\gbb\right)^{N-2}
         \right]^{\frac{2}{N}}}

\nc{\aaa}{\mbox{{\sl a}}}

\nc{\bg}{\mbox{\b{$\gamma$}}}
\nc{\bb}{\mbox{\b{$\beta$}}}
\nc{\zs}{{\frac{\tz}{\ts}}}
\nc{\sz}{{\frac{\ts}{\tz}}}
\nc{\ds}{{\frac{2}{\ts}}}
\nc{\gx}{\bg x^{\zs}}
\nc{\sd}{S_{2}}
\nc{\gxq}{\bg^2 x^{2\zs}}

\nc{\az}{S^{\frac N2}\frac{A(N)}{B(N)}\max_{\partial\Om}H}
\nc{\azz}{A(N)\max_{\partial\Om}H}
\nc{\bd}{\mbox{\b{$\delta$}}}

\nc{\dd}{\frac{1}{2}\alpha\delta}
\nc{\tzud}{\frac{\tz+1}{2}}
\nc{\tzdois}{\frac{\tz}{2}}
\nc{\lb}{\left(}
\nc{\rb}{\right)}
\nc{\defdel}{\frac{\iuz}{\pnhumeio\cdot|u|_{\ts}^{\ts\!/2}}}
\nc{\intum}{\int(\nabla u\cdot\nabla\phi+au\phi)}
\nc{\intdois}{\int u^\tzu\phi}
\nc{\inttres}{\int u^{\tsu}\phi}
\nc{\cs}{{\cal{S}}}
\nc{\sddnd}{S/2^{\frac{2}{N}}}

\newcounter{um}

\def\Nb{{\rm I\kern-.23em N}}
\def\Cb{{\rm\kern.24em
        \vrule width.02em height1.4ex depth-.05ex
        \kern-.26em C}}
\newcommand{\eps}{\varepsilon}

\newtheorem{Thm}{Theorem}[section]
\newtheorem{lemma}[Thm]{Lemma}
\newtheorem{corollary}[Thm]{Corollary}
\newtheorem{proposition}[Thm]{Proposition}

\newtheorem{remark}[Thm]{Remark}

\newcommand{\cal}{\mathcal}

\begin{document}
    

\title[Critical Neumann problem for semilinear equation]{Existence and nonexistence of
    least energy solutions of the Neumann problem for a 
    semilinear elliptic equation with critical Sobolev exponent
    and a critical lower-order perturbation}

\author{David G.\ Costa}

\address{Department of Mathematical Sciences, University of Nevada at 
    Las Vegas\\ Las Vegas, Nevada 89154-4020, USA}

\email{costa@nevada.edu}

\author{Pedro M.\ Gir\~{a}o}

\address{Mathematics Department, Instituto Superior T\'{e}cnico, Av.\ 
    Rovisco Pais, 1049-001 Lisbon, Portugal}

\email{pgirao@math.ist.utl.pt}
      
\thanks{Pedro M.\ Gir\~{a}o was partially supported
    by FCT (Portugal).}



\begin{abstract}
Let $\Om$ be a smooth bounded domain in 
         $\Rb^{N}$, with $N\geq 5$, 
         $a>0$, $\alpha\geq 0$ and $\ts=\frac{2N}{N-2}$. 
         We show that the the 
         exponent $q=\frac{2(N-1)}{N-2}$ plays a critical role 
         regarding the existence of least energy (or ground state) 
         solutions of the Neumann problem
         $$
            \left\{\begin{array}{ll}
            -\Delta u+au=u^\tsu-\alpha u^{q-1}&\mbox{in\ }\Om,\\
            u>0&\mbox{in\ }\Om,\\
            \frac{\partial u}{\partial\nu}=0&\mbox{on\ }\partial\Om.
            \end{array}\right.
        $$
        Namely,
        we prove 
        that when $q=\frac{2(N-1)}{N-2}$
        there exists an $\alpha_{0}>0$
        such that the problem has a least energy 
        solution if $\alpha<\alpha_{0}$ and has
        no least energy solution
        if $\alpha>\alpha_{0}$.
\end{abstract}

\maketitle

\section{Introduction}
  
Let $\Om$ be a smooth bounded domain in $\Rb^{N}$, with $N\geq 5$, 
$a>0$ and $\alpha\geq 0$.  Let $\ts=
\frac{2N}{N-2}$ be the critical exponent for the Sobolev embedding
$H^1(\Om)\subset L^q(\Om)$ and $\tz=\frac{2(N-1)}{N-2}$.
We consider the problem
$$
\left\{\begin{array}{ll}
-\Delta u+au=u^\tsu-\alpha u^{q-1}&\mbox{in\ }\Om,\\
u>0&\mbox{in\ }\Om,\\
\frac{\partial u}{\partial\nu}=0&\mbox{on\ }\partial\Om.
\end{array}\right.\eqno{({\cal{P}}_{\alpha,q})}
$$
We regard $a$ as fixed and $\alpha$ as a parameter.
From Theorem 3.2 of \cite{WX}, due to X.J.\ Wang,
we know that if
$2<q<\tz$, then problem \paq\ has a {\em least energy}\/ solution
for all values of $\alpha\geq 0$.
(Wang's result actually holds for $N\geq 3$.)
A question that naturally arises is the following:  
what happens for $q=\tz$?

It is well known that the solutions of \paq\ correspond to 
critical points of 
the functional $\Phi_{\alpha}:H^1(\Om)\to\Rb$, defined by
$$
\Phi_{\alpha}(u):=\frac 
12\ngupq+\frac{a}{2}\nupq+\frac{\alpha}{q}|u|_{q}^q-\frac 1\ts\ius,
$$
where $|u|_{p}$ denotes the $L^p$ norm of $u$ in $\Om$.  We recall that a 
{\em least energy}\/ solution is a function 
$u\in H^1(\Om)$ such that
$$
\Phi_{\alpha}(u)=\inf_{{\cal{N}}}\Phi_{\alpha}.
$$
The set ${\cal{N}}$ is the Nehari manifold,
${\cal{N}}:=\{u\in H^1(\Om):\Phi_{\alpha}^\prime(u)u=0,$ $u\neq 0\}$.
It is interesting to note that when $q=\tz$ it is possible to determine 
explicitly the function $\Phi_{\alpha}|_{{\cal{N}}}$ by 
solving a quadratic equation. 
We take full advantage of 
this fact.  

We recall that the 
infimum
$$
S:=\inf\left\{\left.
\frac{\int_{\Rb^{N}}|\nabla 
u|^2}{\lb\int_{\Rb^{N}}|u|^\ts\rb^{2/\ts}}
\right|u\in
L^\ts(\Rb^{N}), \nabla u\in L^2(\Rb^{N}),
u\neq 0\right\}
$$
is achieved by the Talenti instanton
$
U(x):=\left(\frac{N(N-2)}{N(N-2)+|x|^2}\right)^{\frac{N-2}{2}}
$.
For $\eps>0$ and $y\in\Rb^{N}$, we define 
$U_{\eps,y}:=\eps^{-\frac{N-2}{2}}U\left(\frac{x-y}{\eps}\right)$. 

Heuristically, we can summarize the main idea behind the analysis of 
problem \paq, when $q=\tz$, as follows.  
There exists an $\alpha_{0}\in]0,+\infty]$ such that
$\inf_{{\cal{N}}}\Phi_{\alpha}<\frac{S^\frac{N}{2}}{2N}$ for 
$\alpha<\alpha_{0}$, and 
$\inf_{{\cal{N}}}\Phi_{\alpha}=\frac{S^\frac{N}{2}}{2N}$
for $\alpha\geq\alpha_{0}$.  If $\alpha<\alpha_{0}$, then \paz\ 
has a least energy solution whereas, if $\alpha>\alpha_{0}$, then \paz\ 
does not have a least energy solution.  
Suppose
$\alpha_{0}=+\infty$, 
so that there exist least energy solutions for all $\alpha\geq 0$.  We
choose a sequence $\ak\to+\infty$ as $k\to+\infty$ and denote by $\uk$ a 
corresponding sequence of least energy solutions.  Then
there would exist a sequence 
of positive numbers $\eps_{k}$ converging to zero, and a sequence of points 
$P_{k}\in\partial\Om$, such
that, modulo a subsequence, $P_{k}\to P$ and
$|\nabla(\uk-U_{\eps_{k},P_{k}})|_{2}\to 0$, as $k\to+\infty$.  
We can use $\Phi_{\ak}(U_{\eps_{k},P_{k}})$ to 
estimate $\Phi_{\ak}(\uk)$ from below with an error that is 
$o(\ak\eps_{k})$.  However, from Adimurthi and Mancini~\cite{AM} and 
X.J.\ Wang~\cite{WX}, we have the estimate
$$
\Phi_{\ak}(U_{\eps_{k},P_{k}})=\frac{S^\frac{N}{2}}{2N}-
\frac{S^\frac{N}{2}}{2}H(P_{k})A(N)\eps_{k}+
\frac{1}{2}B(N)\ak\eps_{k}+o(\ak\ek),
$$
where $A(N)$ and $B(N)$ are positive constants that only depend on 
$N$, and $H(P_{k})$ is the 
mean curvature of $\partial\Om$ at $P_{k}$
with respect to the unit outward normal.  
This lower bound is 
greater than $\frac{S^\frac{N}{2}}{2N}$, for large $k$.  
This contradicts the hypothesis that $\alpha_{0}=+\infty$.

It is somewhat delicate to
justify the use of $\Phi_{\ak}(U_{\eps_{k},P_{k}})$ to 
estimate $\Phi_{\ak}(\uk)$ from below.  
This was first done 
by Adimurthi, Pacella and Yadava in~\cite{APY}, who 
treated the case where $\alpha=0$.  The argument 
involves an expansion to second order of the energy at 
$U_{\eps_{k},P_{k}}$ and a comparison of 
the eigenvalues of the linearized problem at $U_{\eps_{k},P_{k}}$ with 
the eigenvalues of a limiting problem.

The present analysis builds on the work~\cite{APY} of 
Adimurthi, Pacella and 
Yadava, which we will frequently
refer to as [APY].
Of course, the works of Talenti~\cite{T},
Br\'{e}zis and Nirenberg~\cite{BN}, P.L.\ 
Lions~\cite{L1}, Adimurthi and Mancini~\cite{AM}, and X.J.\ 
Wang~\cite{WX}
are also of major importance for our study.

Our main result is the following

\vspace{\baselineskip}

{\sc Theorem.}
{\it
Let $\Om$ be a smooth 
bounded domain in $\Rb^{N}$, with $N\geq 5$, 
$a>0$ and $\alpha\geq 0$.  There exists a positive real number
$\alpha_{0}=\alpha_{0}(a,\Om)$ such that
\begin{enumerate}
    \item[(i)]\ if $\alpha<\alpha_{0}$, then problem \paz\
    has a least energy solution;
    \item[(ii)]\ if 
    $\alpha>\alpha_{0}$, then problem \paz\ has no least energy 
    solution.
    \end{enumerate}}
We remark that 
this result contrasts with Theorem 3.2 of \cite{WX}, referred to above.
Also, from this theorem we deduce an inequality (see (\ref{energia})) which 
implies Aubin's inequality (\ref{cherrier}) (see \cite{Aubin} and
Cherrier~\cite{Ch}). 

We should mention that for any pair $(a,\alpha)$ (with $a>0$ and $\alpha\geq 
0$) problem \paz\ has the constant solution 
$$\textstyle u=\kappa:=\lb\frac{\alpha+
\sqrt{\alpha^2+4a}}{2}\rb^\frac{N-2}{2}.$$
The energy of this solution is 
$$\textstyle\Phi_{\alpha}(\kappa)=
\frac{|\Om|}{\tz N}\left[
\lb\frac{\alpha+\sqrt{\alpha^2+4a}}{2}\rb^{N}+
\frac{\ts}{2\,}a\lb\frac{\alpha+\sqrt{\alpha^2+4a}}{2}\rb^{N-2}
\right],$$
where $|\Om|$ denotes the $N$-dimensional Lebesgue measure of $\Om$.
It follows that for $a>0$ and $\alpha\geq 0$ sufficiently small, namely
for $a\leq {S}/{(2|\Om|)^\frac{2}{N}}$ and $\alpha$ such that
$\Phi_{\alpha}(\kappa)\leq{S^\frac{N}{2}}/{(2N)}$,
then the least energy solutions might be constant.

When the domain $\Om$ is a ball
and $a$ is small, Adimurthi and Yadava~\cite{AY2} proved
that \pazero\ has more than one solution for $N=$ 4, 5 and 6.
However, when  $N=3$ a uniqueness result was proved 
by M.\ Zhu in \cite{Zh2} for
convex domains, $\alpha=0$ and small $a$.

Other works in the spirit of ours are those of 
Br\'{e}zis and Lieb~\cite{BL}, Adimur\-thi and 
Yadava~\cite{AY1}, M.\ Zhu~\cite{Zh}, Z.Q.\ Wang~\cite{W2} and 
Chabrowski and Willem~\cite{CW}.

The organization of this work is as follows.
In Section~2 we give the setup of our work and the 
statement of the main result.
In Section~3 we prove existence of least energy solutions.  
We then assume that the value $\alpha_{0}$ is infinite 
and analyze the asymptotic 
behavior of the least energy solutions as $\alpha\to+\infty$.
In Section~4 we prove nonexistence of least energy solutions.
In Section~5 we give a lower bound for 
    $\alpha_{0}$
    and, using the ideas of Chabrowski and Willem~\cite{CW},
    give partial results concerning existence of 
    least energy solutions for $\alpha=\alpha_{0}$.
In Appendix A we check that the Nehari set 
${\cal{N}}$ is a manifold and a natural constraint for 
$\Phi_{\alpha}$, we derive expressions for $\Phi_{\alpha}|_{{\cal{N}}}$, 
and we derive upper and lower bounds 
for $\Phi_{\alpha}|_{{\cal{N}}}$.
Finally,
in Appendix B we prove a technical estimate, used in our study,
similar to those in 
Adimurthi and Mancini~\cite{AM}.

Motivated by this work, in~\cite{Girao} the second author 
has proved an inequality which improves inequality (\ref{energia}).
In~\cite{PG} he proves a family of inequalities which contains,
as special cases, an inequality in Zhu's work~\cite{Zh} and the 
inequality in~\cite{Girao}.

\section{The setup and statement of the main result}

Let $\Om$ be a smooth bounded domain in $\Rb^{N}$, with $N\geq 5$.  
Let $\ts=
\frac{2N}{N-2}$ be the critical exponent for the Sobolev embedding
$H^1(\Om)\subset L^q(\Om)$ and $\tz=\frac{2(N-1)}{N-2}$.  Finally, 
let
$a>0$ and $\alpha\geq 0$.  We are concerned with the problem of
existence of a {\em least energy\/} solution of
\addtocounter{equation}{+1}
\setcounter{um}{\theequation}
$$
\left\{\begin{array}{ll}
-\Delta u+au=u^\tsu-\alpha u^\tzu&\mbox{in\ }\Om,\\
u>0&\mbox{in\ }\Om,\\
\frac{\partial u}{\partial\nu}=0&\mbox{on\ }\partial\Om.
\end{array}\right.\eqno{(\theequation_{\alpha})}
$$

\nc{\rf}{$(\theum_{\alpha})$}
\nc{\rfk}{$(\theum_{\alpha_{k}})$}
\nc{\rfz}{$(\theum_{\alpha_{0}})$}

Solutions of \rf\ correspond to critical points of 
the functional $\Phi_{\alpha}:H^1(\Om)\to\Rb$ defined by
\be\label{phi}
\Phi_{\alpha}(u):=\frac 12\pnhu+\frac\alpha\tz\iuz-\frac 1\ts\ius.
\ee
We use the notations 
$$|u|_{p}:=\left(\textstyle\int|u|^p\right)^\frac{1}{p}\qquad
\mbox{and}\qquad
||u||:=\left(\ngupq+a\nupq\right)^\frac{1}{2}.$$
Unless otherwise indicated,  
integrals are over $\Om$. 

We recall that the Nehari manifold is
$${\cal{N}}:=\left\{u\in H^1(\Om):\Phi_{\alpha}^\prime(u)u=0, 
u\neq 0\right\}.$$
For any $u\in H^1(\Om)\setminus\{0\}$, there exists a 
unique $t(u)>0$ such that 
$t(u)u\in{\cal{N}}$; the value of $t(u)$ is given in 
expression (\ref{t}) of Appendix A.  
We define 
$\Psi_{\alpha}:H^1(\Om)\setminus\{0\}\to\Rb$ by
$$\Psi_{\alpha}(u):=\Phi_{\alpha}(t(u)u).$$
As can be checked in Appendix B,
\be\label{psiu}
\!\!\!\!\!\!\!\!\Psi_{\alpha}:=\frac 1N\frac 1\tz\frac {1}{2^{N}}
         \left[\left(\gb\right)^{N}
         \!\!+2\cdot\ts\beta\left(\gb\right)^{N-2}
         \right]\!,
\ee
where $\beta$, $\gamma:H^1(\Om)\setminus\{0\}\to\Rb$ are defined by
\be\label{defb}\beta(u):=\frac\nhu\nupsq\ee and
\be\label{defg}\gamma(u)=\gamma_{\alpha}(u):=
\alpha\frac\iuz\nupsz.\ee
Equivalently,
\be\label{psid}
\Psi_{\alpha}=\frac 1N\frac{\beta^\frac N2}\tz
         \left[\left(\delta+\sqrt{\delta^2+1}\right)^{N}+
         \frac\ts{2\;}\left(\delta+\sqrt{\delta^2+1}\right)^{N-2}
         \right],
\ee
with $\beta$ as above and 
$\delta:H^1(\Om)\setminus\{0\}\to\Rb$ defined by
\be\label{delta}
\delta(u)=\delta_{\alpha}(u):=\frac{\gamma(u)}{2\sqrt{\beta(u)}}=
\frac 12\frac{\alpha\iuz}{\pnhumeio\cdot|u|_{\ts}^\frac\ts 
2},
\ee
Obviously, every nonzero critical point of $\Phi_{\alpha}$ is a 
critical point of $\Psi_{\alpha}$.  Since the Nehari manifold is a 
natural constraint for $\Phi_{\alpha}$, if $u$ is a critical point of
$\Psi_{\alpha}$, then 
$t(u)u$ is a critical point of
$\Phi_{\alpha}$.

As is usual, we say that $u\neq 0$ is a {\em ground state\/} 
critical point of $\Phi_{\alpha}$, or a {\em least 
energy\/} solution of 
\rf, if
$$
\Phi_{\alpha}(u)=\inf_{{\cal{N}}}
\Phi_{\alpha}=\inf_{H^1(\Om)\setminus\{0\}}\Psi_{\alpha}.
$$

Our aim is to establish existence and nonexistence of least energy 
solutions of \rf.  We will consider the minimization 
problem corresponding to
$$
S_{\alpha}:=\inf\left\{I_{\alpha}(u)|u\in 
H^1(\Om)\setminus\{0\}\right\},
$$
where $I_{\alpha}:H^1(\Om)\setminus\{0\}\to\Rb$ is defined by
\be\label{defi}
I_{\alpha}:=\left(N\Psi_{\alpha}\right)^\frac 2N.
\ee

From (\ref{psiu}) and (\ref{psid}) we obtain
\be\label{energiad}
I_{\alpha}=\ia
\ee
and
\be\label{iau}
I_{\alpha}=\frac{\beta}{\left(\tz\right)^\frac 2N}
         \left[\left(\delta+\sqrt{\delta^2+1}\right)^{N}+
         \frac\ts{2\;}\left(\delta+\sqrt{\delta^2+1}\right)^{N-2}
         \right]^\frac 2N.
\ee

We observe that 
\be\label{beta}
I_{\alpha}\geq\beta,
\ee
since
$$\frac 1\tz\left(1+\frac\ts{2\;}\right)=1.$$

Before stating our main result, we recall that the 
infimum
$$
S:=\inf\left\{\left.
\frac{\int_{\Rb^{N}}|\nabla 
u|^2}{\lb\int_{\Rb^{N}}|u|^\ts\rb^{2/\ts}}\right|u\in
L^\ts(\Rb^{N}), \nabla u\in L^2(\Rb^{N}),
u\neq 0\right\},
$$
which depends on $N$, is achieved by the Talenti instanton
$$
U(x):=\left(\frac{N(N-2)}{N(N-2)+|x|^2}\right)^{\frac{N-2}{2}}.
$$
This instanton $U$ satisfies 
\begin{equation}\label{inst}
-\Delta U=U^{\ts-1},
\end{equation} so that
\begin{equation}\label{insta}\int_{\Rb^{N}}
    |\nabla U|^2=\int_{\Rb^{N}} U^\ts=S^{\frac{N}{2}}.
\end{equation}
Let $\eps>0$ and $y\in\Rb^{N}$.  For later use, we 
define the rescaled instanton
\be
U_{\eps,y}:=\eps^{-\frac{N-2}{2}}U\left(
\frac{x-y}{\eps}\right),\label{resins}
\ee
which also satisfies (\ref{inst}) and (\ref{insta}).

Our main result is
\begin{Thm}\label{theorem}
    Let $\Om$ be a smooth bounded domain in $\Rb^{N}$, with $N\geq 5$, 
$a>0$ and $\alpha\geq 0$.  There exists a positive real number
$\alpha_{0}=\alpha_{0}(a,\Om)$ such that
\begin{enumerate}
    \item[(i)]\ if $\alpha<\alpha_{0}$, then
    \rf\ has a least energy solution;
    \item[(ii)]\ if 
    $\alpha>\alpha_{0}$, then \rf\ does not have a least energy 
    solution and
    \be\label{energia}
    \sddn\leq
    \frac{\beta}{\left(\tz\right)^\frac 2N}
         \left[\left(\delta+\sqrt{\delta^2+1}\right)^{N}+
         \frac\ts{2\;}\left(\delta+\sqrt{\delta^2+1}\right)^{N-2}
         \right]^\frac 2N
\ee
in $H^1(\Om)\setminus\{0\}$,
where $\beta$ and $\delta=\delta_{\alpha}$ are defined in\/~{\rm
(\ref{defb})}
and\/~{\rm (\ref{delta})}, respectively.  The constant on the 
left hand side of\/~{\rm (\ref{energia})} is sharp.
\end{enumerate}
\end{Thm}

\begin{corollary}[Aubin's inequality]
Let $\Om$ be a smooth bounded domain in $\Rb^{N}$, with $N\geq 5$.  For every
$\varsigma>0$, there exists a $C(\varsigma,\Om)>0$ such that 
\be\label{cherrier}
\sddn-\varsigma\leq\frac{\ngupq+C(\varsigma,\Om)|u|_2^2}{\nupsq},
\ee
for all $u\in H^1(\Om)\setminus\{0\}$.
\end{corollary}
\begin{proof}
    From Lemma~\ref{lemmau},
    there exists a constant 
    $\bar{c}>0$ such that the right hand side of (\ref{energia}) is 
    less than or equal to 
    $$\beta\left(1+\frac{4}{\tz}
    \delta+\bar{c}\delta^2\right)$$ and 
    from H\"older's inequality
    $|u|_{\tz}^\tz\leq |u|_{2}|u|_{\ts}^{{\ts\!/2}}$.  Hence
    $\delta(u)\leq\frac{\alpha}{2}\frac{|u|_{2}}{||u||}$.  
    Let $\epsilon>0$.  For all $u\in H^1(\Om)$,
    \bas
    \sddn|u|_{\ts}^2&\leq&||u||^2\lb 1+\frac{2}{\tz}\alpha_{0}
    \frac{|u|_{2}}{||u||}+\bar{c}\frac{\alpha_{0}^2}{4}
    \frac{|u|_{2}^2}{||u||^2}\rb\\
    &=&||u||^2+\frac{2}{\tz}\alpha_{0}
    ||u||\,|u|_{2}+\bar{c}\frac{\alpha_{0}^2}{4}
    |u|_{2}^2\\
    &\leq&(1+\epsilon)|\nabla u|_{2}^2+
    \lb\frac{\alpha_{0}^2}{(\tz)^2\epsilon}+a\epsilon+
    \bar{c}\frac{\alpha_{0}^2}{4}\rb|u|_{2}^2.
    \eas
\end{proof}
\begin{remark}
    Let $\kappa>0$.  By scaling, we easily check that
    $$\alpha_{0}\lb\kappa^2a,\frac{\Om}{\kappa}\rb=\kappa\,
    \alpha_{0}(a,\Om).$$
\end{remark}

\section{Existence of least energy solutions and their asymptotic behavior}

In this section we start by proving the basic properties of the map
$\alpha\mapsto S_{\alpha}$ and assertion (i) of
Theorem~\ref{theorem}.  We then assume that the value $\alpha_{0}$ in 
Theorem~\ref{theorem} is infinite and analyze the asymptotic 
behavior of the least energy solutions as $\alpha\to+\infty$.

As explained in the previous section, we consider the minimization problem 
corresponding to
$$S_{\alpha}:=\inf\{I_{\alpha}(u)|u\in H^1(\Om), u\not = 0\}.$$
From Adimurthi and Mancini \cite{AM} and X.J.\ Wang~\cite{WX},
we know that 
\be\label{szero}
0<S_{0}<\sddn
\ee
(see (\ref{nukq}) and (\ref{ei}) ahead).  Obviously, $S_{\alpha}$ is 
nondecreasing as $\alpha$ increases.  Choose any point $P\in\partial\Om$.
By testing $I_{\alpha}$ with $U_{\eps,P}$ and letting $\eps\to 0$, we 
conclude that $S_{\alpha}\leq\sddn$ for all $\alpha\geq 0$.

\begin{lemma}\label{pscondition}
    If $S_{\alpha}<\sddn$, then $S_{\alpha}$ is achieved.
\end{lemma}
\begin{proof}
    Let $\uk$ be a minimizing sequence with $\nukps=1$.  Since 
    $\beta\leq I_{\alpha}$, 
    from (\ref{beta}),
    $(\uk)$ is bounded in $H^1(\Om)$.
    We can assume that $\uk\weak u$ in $H^1(\Om)$, $\uk\to u$ a.e.\ on 
    $\Om$, and $|\nabla(\uk-u)|^2\weak\mu$ and
    $|\uk-u|^\ts\weak\nu$ in the sense of measures on $\bar\Om$.  
    Modulo a subsequence, the concentration-compactness lemma 
    implies that
    $$\lim_{k\to\infty}\ngukpq=\ngupq+||\mu||$$
    and
    $$\lim_{k\to\infty}\nukpss=\nupss+||\nu||=1,$$
    where $$\sddn||\nu||^\frac 2\ts\leq||\mu||.$$
    This last inequality is an immediate consequence of
    inequality (\ref{cherrier}).  For 
    $S_{\alpha}=\lim_{k\to\infty}I_{\alpha}(\uk)$,
    we obtain that $S_{\alpha}$ equals $$\iab,$$
    with
    $$\beta_{\infty}=\nhu+||\mu||=
    \frac{\nhu+||\mu||}{(\ius+||\nu||)^{\frac{2}{\ts}}}$$
    and
    $$
    \gamma_{\infty}=\alpha\iuz=\alpha\frac{\iuz}{(\ius+||\nu||)^{\zs}}.
    $$
    If $u=0$, then $$\beta_{\infty}=\frac{||\mu||}{||\nu||^\frac 
    2\ts}\geq\sddn,$$ a contradiction. So $ u\neq 0$. 

    We claim that $||\mu||=0$.  We argue by contradiction and suppose 
    that $||\mu||\ne 0$.  If $||\nu||=0$, then 
    $S_{\alpha}>I_{\alpha}(u)$, which is impossible.  So $||\nu||\neq 
    0$.  
    
    Let $x_{0}:=\ius$, so that $1-x_{0}=||\nu||$.  We define
    $f$, $g$ and $h:[0,1]\to\Rb$ by
    $$f(x):=\gx+\sqrt{\gxq+4\bb x^\ds+4{\textstyle
    \frac{||\mu||}{||\nu||^{2/\ts}}}(1-x)^\ds},$$
    $$g(x):=\bb x^\ds+{\textstyle
    \frac{||\mu||}{||\nu||^{2/\ts}}}(1-x)^\ds$$
    and
    $$h:= f^{N}+2\cdot\ts f^{N-2} g,$$
    for $\bb=\beta(u)$ and $\bg=\gamma(u)$.
    The value $S_{\alpha}$ is
    $$
    S_{\alpha}=\frac{1}{4(\tz)^{\frac{2}{N}}}
    \left[h(x_{0})\right]^{\frac{2}{N}}.$$
    We wish to prove that the minimum of $h$
    occurs at $0$ or $1$.  The former case corresponds to
    $u=0$ and the latter to $||\nu||=0$.  In either case we are led 
    to a contradiction.  This will prove that $||\mu||=0$, thereby 
    establishing the claim.
    
    The derivative of $h$ is
    $$h'=f^{N-3}[N(f^2+4g)f'+2\cdot\ts fg'].$$
    Since
    $$f^2+4g=2f\sqrt{\gxq+4g},$$
    we can write
    $$h'=2f^{N-2}\left[N\sqrt{\gxq+4g}\, f'+\ts g'\right].$$
    The expression for $h'$ can be further simplified by computing 
    $f'$:
    \bas
        \sqrt{\gxq+4g}\, f'&=&\frac{1}{\ts}\left[\tz \bg x^{\zs-1} 
        \sqrt{\gxq+4g}
        +\tz \bg^2 x^{2\zs-1} +4g'\right]\\
        &=&\zs\left[\bg x^{\zs-1} f+2\frac{\ts}{\tz}g'
        \right].
    \eas
    This yields
    \bas
    h'&=&2(N-1)f^{N-2}\left[\bg x^{\zs-1} f+\ts g'
        \right]\\
        &=&2(N-1)f^{N-2}x^{\zs-1}\left[\bg  f+\ts x^{1-\zs} g'
        \right].
    \eas
    We notice that $h'(0)=+\infty$ and $h'(1)=-\infty$; at a zero of 
    $h'$, $g'<0$.

    At a point of minimum of $h$ in the interior of $[0,1]$,
    $h'=0$ and 
    \bas
        \sqrt{\gxq+4g}\, f'&=& \zs\left[\bg x^{\zs-1} f+\ts g'
        -\ts\left(1-\frac{2}{\tz}\right)g'\right]\\
        &=&-(\tz-2)g';
    \eas  
    we notice that at a zero of 
    $h'$, $f'>0$.
    
    We consider
    $$\kappa:=-\ts x^{1-\zs} g',$$
    whose derivative is
    \bas\kappa'&=&-(\ts-\tz)x^{-\zs}g'-\ts 
    x^{\frac{\ts-\tz}{\ts}}g^{\prime\prime}\\
    &>&-(\ts-\tz)x^{-\zs}g'\\
    &=&x^{-\zs}\sqrt{\gxq+4g}\,f'\mbox{\ \ \ for $h'=0$}\\
    &>&\bg f'.    
    \eas
    
    The zeros of $h'$ occur when $\bg f=\kappa$.  We just proved 
    that $\kappa'>\bg f'$ at the zeros of $h'$.  This implies that the 
    graphs of $\bg f$ and $\kappa$ can cross at most once, and that 
    $h'$ has at most one zero.  If the function $h$ were to have a 
    minimum in the interior of $[0,1]$, then $h'$ 
    would have at least  three zeros 
    because $h'(0)=+\infty$ and $h'(1)=-\infty$.  We conclude that $h$ 
    has no minimum inside $[0,1]$.  (The conditions on the 
    derivative of $h$ at the end points of the interval,
    or the fact that the graphs of $\bg f$ and $\kappa$ cross,
    imply that $h'$ does vanish 
    inside $[0,1]$, at a point of maximum of $h$.)    
    Therefore the minimum of $h$ occurs either at $0$ or $1$
    and we have proved our claim. 
    
    Since $||\mu||=0$, the function $u$ is a minimizer for 
    $I_{\alpha}$.
    \end{proof}

\begin{lemma}\label{cont}
    The map $\alpha\mapsto S_{\alpha}$ is continuous on $[0,+\infty[$.
\end{lemma}
\begin{proof}
    Let $\bar\alpha\in[0,+\infty[$.  First we prove that 
    $\alpha\mapsto S_{\alpha}$ is continuous from the right at 
    $\bar\alpha$.
    If $S_{\bar\alpha}=\sddn$, then 
    continuity from the right at $\bar\alpha$
    is obvious.
    If $S_{\bar\alpha}<\sddn$, let $u_{\bar\alpha}$ be a minimizer of
    $I_{\bar\alpha}$, which exists by the previous lemma.  If 
    $\alpha>\bar\alpha$, then $S_{\bar\alpha}\leq S_{\alpha}\leq
    I_{\alpha}(u_{\bar\alpha})\to S_{\bar\alpha}$ as 
    $\alpha\searrow\bar\alpha$.  This proves continuity from the right 
    at $\bar\alpha$.
    
    To prove continuity from the left  we show that 
    $\lim_{\alpha\nearrow\bar\alpha}S_{\alpha}=S_{\bar\alpha}$.
    If  the value of the limit on the left hand side is $\sddn$, then this 
    equality is obvious.  So suppose
    $\lim_{\alpha\nearrow\bar\alpha}S_{\alpha}<\sddn$.
    Choose a sequence $\ak\nearrow\bar\alpha$ and
    $\uk\in H^1(\Om)$, with $\nukps=1$, such that $I_{\ak}(\uk)=S_{\ak}$.
    By (\ref{beta}), the sequence $(\uk)$ is bounded in $H^1(\Om)$ and 
    we can assume 
    that $\uk\weak u$ in $H^1(\Om)$.  
    An application of the concentration-compactness principle,  
    as in the previous lemma, shows that $u\neq 0$ and
    $$\lim_{k\to\infty}I_{\ak}(\uk)\geq I_{\bar\alpha}(u).$$
    So,
    $$S_{\bar\alpha}\leq I_{\bar\alpha}(u)\leq\lim_{k\to\infty}I_{\ak}(\uk)
    =\lim_{\ak\nearrow\bar\alpha}S_{\alpha}$$
    and $S_{\bar\alpha}=\lim_{\ak\nearrow\bar\alpha}S_{\alpha}$.
\end{proof}

By the previous lemma, the value
\be\label{min}
\alpha_{0}:=
\left\{\begin{array}{l}
+\infty,\mbox{\ if\ }S_{\alpha}<\sddn\mbox{\ for all\ 
}\alpha\in[0,+\infty[,\\
\min\left\{\alpha\in[0,+\infty[\;\left|\;S_{\alpha}=\sddn\right.\right\},
\mbox{\ otherwise.}
\end{array}\right.
\ee
is well defined.  By (\ref{szero}) it is not zero.
Lemma~\ref{pscondition} implies the following two corollaries:
\begin{corollary}
    The map $\alpha\mapsto S_{\alpha}$ is strictly increasing on 
    $[0,\alpha_{0}]$.
\end{corollary}
\begin{corollary}\label{coro}
    If $\alpha\in[0,\alpha_{0}[$, then \rf\ has a least energy solution 
    $u_{\alpha}$.  If $\alpha\in]\alpha_{0},+\infty[$, then \rf\ does 
    not have a least energy solution.
\end{corollary}

This proves (i) of Theorem~\ref{theorem}.  Assertion (ii) of 
Theorem~\ref{theorem}
will also follow once we establish that $\alpha_{0}$ is finite.  

\begin{lemma}\label{infinito}
    If $S_{\alpha}<\sddn$ for all $\alpha\geq 0$, then 
    \be\label{salpha} \lim_{\alpha\to+\infty}S_{\alpha}=\sddn
    .\ee  
    Suppose $\alpha_{k}\to+\infty$ as $k\to+\infty$ and
    $\uk$ is a minimizer for $I_{\alpha_{k}}$ satisfying 
    \rfk.  Then $\uk\weak 0$
    in $H^1(\Om)$ and
    $$\mk:=\max_{\bar\Om}\uk$$
    converges to $+\infty$, as $k\to\infty$.
\end{lemma}
\begin{proof}
    Suppose $S_{\alpha}<\sddn$ for all $\alpha\geq 0$ and choose any
    sequence $\alpha_{k}\to+\infty$ as $k\to+\infty$.  Let $\uk$ 
    be a minimizer for $I_{\alpha_{k}}$ satisfying 
    \rfk, which necessarily exists by 
    Lemma~\ref{pscondition} and rescaling. 
    We claim that $\uk$ is bounded in $H^1(\Om)$.  Indeed,
    by (\ref{energiad}), 
    $$
    \frac{1}{(\tz)^\frac 2N}\;
    \gamma^2(u_{k})\leq I_{\alpha}(u_{k})\leq\sddn.
    $$
    So, 
    $$\ak|\uk|_{\tz}^\tz\leq\left(\frac\tz 2\right)^\frac 
    1NS^\frac{1}{2}|\uk|_{\ts}^\tz.$$
    By \rfk,
    \be\label{eqn}
    \nukpss=\nhukpq+\ak\iukpz.
    \ee
    Together,
    \bas
    |\uk|_{\ts}^{\ts-2}&\leq&\beta(u_{k})+\left(\frac\tz 2\right)^\frac 
    1NS^\frac{1}{2}|\uk|_{\ts}^{\tz-2}\\
    &\leq&\sddn
    +\left(\frac\tz 2\right)^\frac 
    1NS^\frac{1}{2}|\uk|_{\ts}^{\tz-2},
    \eas
    since, by (\ref{beta}), $\beta(u_{k})
    \leq I_{\alpha_{k}}(u_{k})\leq\sddn$.
    So $\nukps$ is bounded.  Recalling that
    $\beta(u_{k})=\frac{\nhukpq}{\nukpsq}$,
    we conclude that $\uk$ is bounded in $H^1(\Om)$.
    
    From (\ref{eqn}), we conclude that $\uk\weak 0$ in $H^1(\Om)$.
    We can assume that $\uk\to 0$ a.e.\ on 
    $\Om$, and $|\nabla\uk|^2\weak\mu$ and
    $|\uk|^\ts\weak\nu$ in the sense of measures on $\bar\Om$.  
    Then
    \be\label{limu}\lim_{k\to\infty}\ngukpq=||\mu||\ee
    and
    \be\label{limd}\lim_{k\to\infty}\nukpss=||\nu||,\ee
    where \be\label{limt}\sddn||\nu||^\frac 2\ts\leq||\mu||.\ee
    Thus
    \be\label{auxu}
    \sddn\geq \lim_{k\to\infty}S_{\ak}=\lim_{k\to\infty}
    I_{\alpha_{k}}(u_{k})\geq\frac{||\mu||}
    {||\nu||^\frac 2\ts}\geq\sddn,\ee
    and the inequalities in (\ref{auxu}) are equalities. 
    This proves (\ref{salpha}).  
    
    From \rfk, the values $M_{k}$ satisfy
    \be\label{mmkk}a+\ak M_{k}^\tzd\leq M_{k}^\tsd\ee
    and consequently $\mk\to +\infty$ as $k\to+\infty$.  
\end{proof}

\begin{lemma}\label{instantao}  
       Let $S_{\ak}<\sddn$ and 
       $S_{\ak}\to\sddn$ as $\alpha_{k}\to\alpha_{0}\in]0,+\infty]$.
       Denote by $\uk\in H^1(\Om)$ a minimizer for $I_{\alpha_k}$ 
       satisfying \rfk.  
       In case $\alpha_{0}<+\infty$ suppose  
       that $\uk\weak 0$.
       Then
       \be\label{sndd}
       \lim_{k\to\infty}\ngukpq=\lim_{k\to\infty}\nukpss=\sndd.
       \ee
       Moreover, if $\alpha_{0}=+\infty$, or
       if $\alpha_{0}<+\infty$ and we further assume that
       $\lim_{\alpha_{k}\to\alpha_{0}}M_{k}=+\infty$,
       then we also have
       \be\label{akdk}
       \lim_{k\to\infty}\ak\delta_{k}=0,\ee
       \be\label{convum}
       \lim_{k\to\infty}|\nabla\uk-\nabla\ium|_{2}=0
       \ee
       and $\pkk\in\partial\Om$, for large $k$.
       Here, we are denoting
       $$
       \delta_{k}:=M_{k}^{-\frac{2}{N-2}},
       $$
       and $P_{k}$ is such that
       $M_{k}=\uk(\pkk)$.
\end{lemma}
     {\em Note.\/}  If $\alpha_{0}=+\infty$, Lemma~\ref{infinito} 
     guarantees the conditions
     $S_{\ak}\to\sddn$, $\uk\weak 0$ and $M_{k}\to+\infty$ are 
     satisfied.\\
\begin{proof}  By \rfk, $\uk$ satisfies (\ref{eqn}).
    Since $\uk\weak 0$ in $H^1(\Om)$, $\uk$ is bounded in 
    $H^1(\Om)$.  Therefore, (\ref{limu}), (\ref{limd}), (\ref{limt}) 
    and
    (\ref{auxu}) hold, with equalities in (\ref{auxu}).
    Hence, $\beta(u_{k})\to\sddn$.  From (\ref{iau}), 
    $\delta(u_{k})\to 0$ and
    $$\lim_{k\to\infty}\ak|\uk|_{\tz}^\tz=0,$$  
    as $\uk$ is bounded in 
    $H^1(\Om)$.
    Taking limits 
    in (\ref{eqn}) as $k\to\infty$, 
    \be\label{auxd}||\nu||=||\mu||.\ee
    Combining (\ref{auxu}) and (\ref{auxd}), equalities 
    (\ref{sndd}) follow.   

     We now use the Gidas and Spruck blow up technique~\cite{GS}.
     Let $v_{k}(x):=\delta_{k}^\ndd \uk(\dk x+\pkk)$ for $x\in\Om_{k}:=
     (\Om-\pkk)/\dk$, so that 
     $$
     \left\{\begin{array}{ll}
     -\Delta\vk+a\delta_{k}^2\vk+\ak\dk v_{k}^\tzu=v_{k}^\tsu&\mbox{in\ 
     }\Om_{k},
     \\
     0<\vk\leq\vk(0)=1&\mbox{in\ }\Om_{k},\\
     \frac{\partial \vk}{\partial\nu}=0&\mbox{on\ }\partial\Om_{k}.
     \end{array}\right.
     $$
     
     Rewriting 
     (\ref{mmkk}) in 
     terms of the $\dk$,
     $$
     a\delta_{k}^2+\ak\dk\leq 1.
     $$
     So, we can assume that $P_{k}\to P_{0}$,
     $$\mbox{dist\,}(P_{k},\partial\Om)/\dk\to L\in[0,+\infty],$$ $$
     \ \Om_{k}\to\Om_{\infty}:=\{(\tilde x,x_{N})
     \in\Rb^{N-1}\times\Rb:x_{N}>-L\}$$
     and 
     $\ak\dk\to\bar\alpha$.
     By the elliptic estimates in~\cite{ADN}, 
     \be\label{ctloc}\vk\to v\mbox{\ in\ } 
     C^2_{\mbox{\tiny loc}}(\Om_{\infty})\ee
     where $v$ satisfies
     $$
     \left\{\begin{array}{ll}
     -\Delta v+\bar\alpha v^\tzu=v^\tsu&\mbox{in\ 
     }\Om_{\infty},
     \\
     0<v\leq v(0)=1&\mbox{in\ }\Om_{\infty},\\
     \frac{\partial v}{\partial\nu}=0&\mbox{on\ }\partial\Om_{\infty}
     \end{array}\right.
     $$
     as $a\delta_{k}^2\to 0$.
     By lower 
     semicontinuity of the norm, $v\in L^\ts(\Om_{\infty})$ and
     $\nabla v\in L^2(\Om_\infty)$.  
     So, we can apply
     Pohozaev's identity and get $\bar\alpha=0$, and thus
     $v=U$.  

     If $L=+\infty$, then $\Om_{\infty}=\Rb^{N}$. From (\ref{sndd}),
     $$
     S^\frac N2=\int_{\Rb^{N}}|\nabla 
     U|^2\leq\lim_{k\to\infty}|\nabla\uk|_{2}^2=\sndd,$$
     which is impossible.  

     So $L$ is finite.  This implies that $P_{0}\in\partial\Om$.
     In fact, $L$ has to be zero since $v\leq v(0)$.  
     Using a diffeomorphism to straighten a boundary portion of $\Om$,
     the argument in Lemma 2.2 of [APY] shows that
     $\pkk\in\partial\Om$ for large $k$.  Finally, from (\ref{sndd}), 
     (\ref{ctloc}) and
     $$
     \int_\ttt |\nabla U|^2=\sndd,$$ we deduce (\ref{convum}).
\end{proof}

As in \cite{APY} and \cite{BE},
let $${\cal{M}}:=\{CU_{\eps,y}, C\in\Rb, \eps>0, y\in\partial\Om\}$$
and $d(u,{\cal{M}}):=\inf\{|\nabla (u-V)|_{2}, V\in{\cal{M}}\}$.
The set ${\cal{M}}\setminus\{0\}$ is a manifold of dimension $N+1$.
The tangent space 
$T_{C_{l},\eps_{l},y_{l}}({\cal{M}})$ at $C_{l}U_{\eps_{l},y_{l}}$ is 
given by
$$
T_{C_{l},\eps_{l},y_{l}}({\cal{M}})=\mbox{span\,}\left\{
U_{{\eps},y},C\frac\partial{\partial\eps}U_{{\eps},y},
C\frac\partial{\partial\tau_{i}}U_{{\eps},y}, 1\leq i\leq N-1
\right\}_{(C_{l},\eps_{l},y_{l})}
$$
where $T_{x}(\partial\Om)=\mbox{span}\{\tau_{1},\ldots,\tau_{N-1}\}$.

For large $k$,
the infimum $d(\uk,{\cal{M}})$ is achieved: 
\be\label{m}
d(\uk,{\cal{M}})=|\nabla(\uk-C_{k}U_{\eps_{k},y_{k}})|_{2}
\mbox{\ for\ }
\ck U_{\eps_{k},y_{k}}\in{\cal{M}}.
\ee
Furthermore,
\be\label{ck}
\ck=1+o(1)
\ee
$y_{k}\to P_{0}$ and $\eps_{k}/\delta_{k}\to 1$ (see Lemma 1 of 
\cite{BE} and Lemma 2.3 of \cite{APY}).  From (\ref{akdk}),
\be\label{akek}
\ak\ek\to 0.
\ee

We define $$\wk:=\uk-\ck U_{\eps_{k},y_{k}},$$ 
so that
\be\label{dot}
\int\nabla
U_{\eps_{k},y_{k}}\cdot\nabla w_{k}=0.
\ee
Now, on the one hand, from (\ref{convum}), 
$$
\lim_{k\to\infty}|\nabla(\uk-\ck U_{\eps_{k},y_{k}})|_{2}=0.
$$
On the other hand, from 
Poincar\'{e}'s inequality, and the fact that both the average of
$\uk$ and the average of $\ck U_{\eps_{k},y_{k}}$, 
in $\Om$, converge to zero,
$$
\lim_{k\to\infty}|\uk-\ck U_{\eps_{k},y_{k}}|_{\ts}=0.
$$
Together,
\be\label{zero}
\lim_{k\to\infty}||w_{k}||=0.
\ee

Our next objective is the upper bound in Lemma~\ref{dif} for 
$\int U_{\eps_{k},y_{k}}^{\ts\!-2}\, w_{k}^2$
in terms of 
$|\nabla\wk|_{2}^2+(\tz-1)\ak
    \int U_{\eps_{k},y_{k}}^{\tz\!-2}\, w_{k}^2$.
This will be crucial in the lower bound estimates for the energy in 
Section~\ref{four}.  

The eigenvalue 
problems arising from the linearization of \rfk\ at 
$U_{\eps_{k},y_{k}}$ are related to the
eigenvalue problem in
\begin{lemma}[Bianchi and Egnell~\cite{BE}, Rey~\cite{R}] The 
eigenvalue problem\label{ber} 
$$\left\{
\begin{array}{ll}
    -\Delta \phi=\mu U^\tsd\phi&\mbox{in\ }\ttt,\\
    \frac{\partial\phi}{\partial\nu}=0&\mbox{on\ }\partial\ttt,\\
\int_{\ttt}U^\tsd\phi^2<\infty&
\end{array}\right.
$$
admits a discrete spectrum $\mu_{1}<\mu_{2}
\leq\mu_{3}\leq\ldots$ such 
that $\mu_{1}=1$, $\mu_{2}=\mu_{3}=\ldots=\mu_{N}=\tsu$ and 
$\mu_{N+1}>\tsu$. The eigenspaces $V_{1}$ and $V_{(\tsu)}$, 
corresponding to 1 and $(\tsu)$, are given by
\bas
 V_{1}&=&\mbox{span\ }U,\\
 V_{(\tsu)}&=&\textstyle\mbox{span}\left\{\left.
 \frac{\partial U_{1,y}}{\partial y_{i}}
 \right|_{y=0},\mbox{\ for\ }1\leq i\leq N-1
 \right\}.
\eas
\end{lemma}

We will consider the eigenvalue 
problems arising from the linearization of \rfk\ at 
$U_{\eps_{k},y_{k}}$.
Let $\eps>0$, $\nu_{\eps}>0$, and $y_{\eps}\in\partial\Om$ with 
$\lim_{\eps\to 0}y_{\eps}=y_{0}$.  Let $\{\phi_{i,\eps}\}_{i=1}^\infty$
be a complete set of orthogonal eigenfunctions with eigenvalues
$\mu_{1,\eps}<\mu_{2,\eps}\leq\mu_{3,\eps}\leq\ldots$ for the weighted
eigenvalue problem
$$\left\{
\begin{array}{ll}
    -\Delta \phi+\nu_{\eps}U^\tzd_{\eps,y_{\eps}} \phi=\mu 
    U_{\eps,y_{\eps}}^\tsd\phi&\mbox{in\ }\Om,\\
    \frac{\partial\phi}{\partial\nu}=0&\mbox{on\ }\partial\Om,\\
\end{array}\right.
$$
with $\phi_{1,\eps}>0$ and
$$
\int_{\Om}U^\tsd\phi_{i,\eps}\phi_{j,\eps}=\delta_{i,j}.
$$
Let
$$
\Om_{\eps}:=(\Om-y_{\eps})/\eps.
$$
The sets $\Om_{\eps}$ converge to a half space as $\eps\to 0$.
For a function $v$ on $\Om$, we 
define $\tilde v$ on $\Om_{\eps}$ by
$$
\tilde v(x):=\eps^{\frac{N-2}2}v(\eps x+y_{\eps}).
$$

The relation between these eigenvalue problems and the one considered 
in Lemma~\ref{ber} is given in

\begin{lemma}\label{ttt}
    Suppose $y_{\eps}\in\partial\Om$, $\lim_{\eps\to 0}y_{\eps}=y_{0}$,
    $\lim_{\eps\to 0}(\eps\nu_{\eps})=0$ and
    the sets\/ $\Om_{\eps}$ converge to $\ttt$.
    Then, up to a subsequence, 
    $$
    \lim_{\eps\to 0}\mu_{i,\eps}=\mu_{i}
    $$
    and
    $$
    \lim_{\eps\to 0}\int_{\Om_{\eps}}U^\tsd(\tilde{\phi}_{i,\eps}
    -\tilde{\phi}_{i})^2=0.
    $$
    The $\mu_{i}$ and $\tilde{\phi}_{i}$ satisfy
    $$\left\{
    \begin{array}{ll}
    -\Delta \tilde{\phi}_{i}=\mu_{i} U^\tsd\tilde{\phi}_{i}&\mbox{in\ 
    }\ttt,\\
    \frac{\partial\tilde{\phi}_{i}}{\partial\nu}=0&\mbox{on\ 
    }\partial\ttt,\\
    \int_{\ttt}U^\tsd\tilde{\phi}_{i}^2=1,&
    \end{array}\right.
    $$
    and the functions $\tilde{\phi}_{i}$ are supposed 
    extended to $\Rb^{N}$ by 
    reflection.
    In particular, from the previous lemma, $\mu_{1}=1$, 
    $\tilde{\phi}_{1}=CU$ for some constant $C>0$,  
    $\mu_{i}=\tsu$ for $2\leq i\leq 
    N$ and $\mu_{N+1}>\tsu$.  Also, 
    $\{\tilde{\phi}_{i}\}_{i=2}^{N}$ is in the span of
    $\left\{{\partial U_{1,y}}/{\partial y_{i}}
    \right|_{y=0},\mbox{\ for\ }1\leq i\leq N-1\}$.
\end{lemma} 
The proof of Lemma~\ref{ttt} is a consequence of the arguments in the 
proof of Lemma~3.3 of [APY], of Lemma~\ref{nove} 
and of Remark~\ref{dez}.  For the details we refer to the proof of
Lemma~5.5 of \cite{PG} for parameter $s$ there
equal to one.
\begin{lemma}\label{nove} Suppose $y_{\eps}\in\bar\Om$,
     $\phi_{\eps}\in H^1(\Om)$,
    $$\left\{
    \begin{array}{l}\displaystyle
    \int U^\tzd_{\eps,y_{\eps}}\phi_{\eps}^2\to 0,\\ \\
    \displaystyle
    \int|\nabla\phi_{\eps}|^2\to 0,
    \end{array}\right.
    $$
    as $\eps\to 0$.
Then
$$
    \int U^\tsd_{\eps,y_{\eps}}\phi_{\eps}^2\to 0,
$$ 
    as $\eps\to 0$.
\end{lemma}
\begin{proof}
    We denote the average of
    $\pk$ in $\Om$ by $\pb$.  By Poincar\'{e}'s inequality,
    $$|\pk-\pb|_{\ts}\to 0.
    $$
    The limits in this proof are taken as $\eps\to 0$.
    So we can write $\pk=\pb+\et$, with $\et\to 0$ in $L^\ts$.
    We know that 
    $$\int\uzz(\pbq+2\pb\et+\etq)=o(1).$$  
    We have the following estimates
    for the three terms on the 
    left hand side:
    $$
    \int\uzz\pbq\geq b\pbq\eps,
    $$
    for some $b>0$, and 
    \bas
    \left|\int\uzz\et\pb\right|&\leq&|\et|_{\ts}|\pb|\left(\int 
    U_{\eps,y_{\eps}}^{{\frac 
    {N}{N-2}}{\frac{4}{N+2}}}\right)^\frac {N+2}{2N}\\
    &\leq& C|\et|_{\ts}|\pb|\eps,
    \eas
    by (\ref{um}); and
    \ba
    \int\uzz\etq&\leq&|\et|_{\ts}^2\left(\int U_{\eps,y_{\eps}}^{\frac 
    {N}{N-2}}\right)^\frac 2N\nonumber\\
    &\leq& C|\et|_{\ts}^2\eps|\log\eps|^\frac 2N,\label{iuzd}
    \ea
    by (\ref{dois}). (Inequalities (\ref{um}), (\ref{dois}) and 
    (\ref{tres}) are in the beginning of the next section.)
    Thus
    $$
    b\pbq\eps\leq C|\pb|\eps+o(1).$$
       This shows that $\pb\sqrt{\eps}$ is bounded.  
       But if $\pb\sqrt{\eps}$ is 
    bounded this shows that 
    \be\label{pe}\pb\sqrt{\eps}\to 0.\ee 
    
    We want to prove that 
    $$\int\uss(\pbq+2\pb\et+\etq)=o(1).$$  For the first term on the 
    left hand side we have, by (\ref{um}) and then (\ref{pe}),
    $$
    \int\uss\pbq\leq C\pbq\eps^2\to 0.
    $$  
    For the third term we have
    $$
    \int\uss\etq\leq C|\et|_{\ts}^2\to 0.
    $$
       We claim that the remaining term also converges to zero.  This 
    will prove the lemma.  For the second term we have the estimate
    $$
    \zeta_{\eps}:=
    \left|\int\uss\pb\et\right|\leq|\et|_{\ts}|\pb|\left(\int
    U_{\eps,y_{\eps}}^{\frac 
    N{N-2}\frac{8}{N+2}}\right)^{\frac{N+2}{2N}}.
    $$
    If $N=5$, by (\ref{tres}),
    $$
    \zeta_{\eps}\leq C|\et|_{\ts}|\pb|\eps^{N\left(1-{\frac 
    {4}{N+2}}\right){\frac{N+2}{2N}}}\leq 
    C|\et|_{\ts}|\pb|\eps^{\frac{N-2}{2}}=
    C|\et|_{\ts}|\pb|\eps^{\frac 32}.
    $$
    If $N=6$, by (\ref{dois}),
    $$
    \zeta_{\eps}\leq C|\et|_{\ts}|\pb|\eps^2|\log\eps|^\frac{2}{3}.
    $$
    Finally, if $N\geq 7$, by (\ref{um}),
    $$\zeta_{\eps}\leq C|\et|_{\ts}|\pb|\eps^2.$$ 
    In all three cases, (\ref{pe}) implies that $\zeta_{\eps}\to 0$.
    \end{proof}
\begin{remark}\label{dez}
    If in the previous lemma, instead of assuming
    $\int|\nabla\phi_{\eps}|^2\to 0$, we assume that
    $\int|\nabla\phi_{\eps}|^2$ is bounded, then we can still conclude
    $\bar{\phi}_{\eps}\sqrt{\eps}\to 0$ and
    $\int U_{\eps,y_{\eps}}^{\ts-2}\lb\pbq+2\pb\eta_{\eps}
    \rb=
    \int_{\Om_{\eps}} U^{\ts-2}\lb
    \tilde{\bar{\phi}}_{\eps}^2+2\tilde{\bar{\phi}}_{\eps}\tilde\eta_{\eps}
    \rb
    \to 0$, as $\eps\to 0$.
\end{remark}
    
Using Lemma~\ref{ttt} and the arguments in the proof of Lemma 3.4 of [APY],
we deduce
\begin{lemma}\label{dif}
    Suppose $y_{\eps}\in\partial\Om$, $\lim_{\eps\to 0}y_{\eps}=y_{0}$
    and
    $\lim_{\eps\to 0}(\eps\nu_{\eps})=0$.
    There 
    exists a constant $\gamma_{1}>0$ such that, for sufficiently 
    small $\eps$,
    $$|\nabla w|_{2}^2+\nu_{\eps}
    \int U_{\eps,y_{\eps}}^\tzd w^2
    \geq(\tsu+\gamma_{1})
    \int U_{\eps,y_{\eps}}^\tsd w^2+O(\eps^2||w||^2)
    $$
    for $w$ orthogonal to $T_{1,\eps,y_{\eps}}({\cal{M}})$.
\end{lemma}  

\section{Nonexistence of least energy solutions}\label{four}

In this section we prove (ii) of Theorem~\ref{theorem}.
The idea of the proof is to obtain a lower bound for $I_{\alpha}$ and show 
that if $\alpha_{0}$, defined in (\ref{min}),
is infinite, then the least energy 
solutions $u_{k}$ of \rfk\ have energy $I_{\ak}(\uk)>
\sddn$, for large $\alpha_k$.  This is impossible.  
Therefore $\alpha_{0}$ is finite.
By Corollary~\ref{coro}, (ii) of Theorem~\ref{theorem} 
follows.

Assume $$\uk=C_{k}U_{\eps_{k},y_{k}}+w_{k},$$ (\ref{sndd}),
(\ref{m}), (\ref{ck}), (\ref{akek}) and (\ref{zero}).
  From 
    (\ref{iau}), $I_{\alpha}$ has the lower bound
    \be\label{lb}
    I_{\alpha}\geq \beta\left(1+\frac 4\tz\delta\right)
    \ee
    (this is also checked in (\ref{lbd}) of Appendix A).
We will expand $\beta$ and $\delta$ 
to second order around $U_{\eps_{k},y_{k}}$.  We start by 
deriving estimates for the 
terms that appear in this expansion.

We recall, from Br\'{e}zis and Nirenberg~\cite{BN}, that, for 
$y\in\bar\Om$, there exist positive
constants $c_{1}$ and $c_{2}$ such that:\\
if $1\leq q<\frac N{N-2}$, then
\be\label{um} c_{1}\eps^{q\left(\frac{N-2}{2}\right)}
\leq |U_{\eps,y}|_{q}^q\leq
c_{2}\eps^{q\left(\frac{N-2}{2}\right)};\ee
if $q=\frac N{N-2}$, then
\be\label{dois} c_{1}\eps^{\frac{N}{2}}|\log\eps|
\leq |U_{\eps,y}|_{q}^q\leq
c_{2}\eps^{\frac{N}{2}}|\log\eps|;\ee
and if $\frac N{N-2}<q\leq\frac {2N}{N-2}$, then
\be\label{tres} c_{1}\eps^{N\left(1-\frac{q}{\ts}\right)}\leq 
|U_{\eps,y}|_{q}^q\leq
c_{2}\eps^{N\left(1-\frac{q}{\ts}\right)}.\ee

For brevity, we shall write $$U_{k}:=U_{\eps_{k},y_{k}}.$$

{\em Estimate for $\nukq$:\/}  For $N\geq 5$ , $\frac N{N-2}<2$.  From 
(\ref{tres}), 
\be\label{nukq}\nukq=O(\ek^2).\ee

{\em Estimate for $\iukz$:\/} Since $y_{k}\in\partial\Om$
and we are supposing that the domain is smooth,
\ba\iukz 
&=&\frac{\tz B(N)\ek}2+o(\ek),\label{iukz}
\ea
with
\bas
B(N)&=&
\frac 1\tz\int_{\Rb^{N}}U^\tz\\
&=&\omega_{N}\frac{1}{2^{N}}\sqrt{\pi}\,
\frac{\Gamma\left(\frac{N}{2}\right)}{\Gamma\left(\frac{N+1}{2}\right)}
[N(N-2)]^\frac{N}{2},
\eas
as proved in Appendix B. Here $\omega_{N}$
is the volume of the $N-1$ dimensional unit sphere.

{\em Estimate for $\ngukq$ and for $\iuks$:\/} From Adimurthi and 
Mancini \cite{AM}, since $N\geq 5$,
\be\label{ngukq}
    \ngukq=\sndd-\bar{C}_{1}\ek+O(\ek^2)
\ee
and
\be\label{iuks}
    \iuks=\sndd-\bar{C}_{2}\ek+O(\ek^2),
\ee
where
$$
\bar{C}_{1}=H(y_{k})
\frac{\omega_{N-1}(N-2)^2}
{4}
\frac{\Gamma\left(\frac{N+3}2\right)
\Gamma\left(\frac{N-3}2\right)}{\Gamma(N)}[N(N-2)]^{\frac{N-2}2}
$$
and
$$
\bar{C}_{2}=H(y_{k})\frac{\omega_{N-1}}
{4}\frac{\Gamma\left(\frac{N+1}2\right)
\Gamma\left(\frac{N-1}2\right)}{\Gamma(N)}[N(N-2)]^{\frac N2}.
$$
Here $H(y_{k})$ denotes the mean curvature
of $\partial\Om$  at $y_{k}$ 
with respect to the unit outward normal
and, as above, $\omega_{N}$
is the volume of the $N-1$ dimensional unit sphere.
This yields
\be\label{ei}
\frac{\ngukq}{\nuksq}=\frac{S}{2^\frac 
2N}-2^{\frac{N-2}N}SH(y_{k})
A(N)\ek+O(\eps_{k}^2)
\ee
with
\be\label{aden}
A(N)=
\frac{2}{N}\frac{\omega_{N-1}}{\omega_{N}}
\frac{\Gamma\left(\frac{N+1}2\right)\Gamma\left(\frac{N-3}2\right)}
{\Gamma\left(\frac{N}2\right)\Gamma\left(\frac{N-2}2\right)}
=
{\textstyle
\frac{N-1}{N}}\frac{1}{\sqrt{\pi}}\,
\frac{\Gamma\left(\frac{N-3}{2}\right)}{\Gamma\left(\frac{N-2}{2}\right)}.
\ee
To justify the last equality we recall that if $\omega_{N}(r)$ 
is the volume 
of the $N-1$ dimensional sphere with radius $r$, then
$$
\omega_{N}(r)=\int_{0}^\pi\omega_{N-1}(r\sin\phi)\,rd\phi=
r^{N-1}\omega_{N-1}(1)\int_{0}^\pi\sin^{N-2}\phi\,d\phi,
$$
which yields
$$
\frac{\omega_{N-1}}{\omega_{N}}=
\frac{\omega_{N-1}(1)}{\omega_{N}(1)}=
\frac{1}{\sqrt{\pi}}\,
\frac{\Gamma\left(\frac{N}{2}\right)}{\Gamma\left(\frac{N-1}{2}\right)}.
$$
We mention that the Talenti instanton 
we use does not coincide with the one in 
\cite{AM}.  Denoting the Talenti instanton in Adimurthi and Mancini by $V$,
$V(\ \cdot\ )=U((N(N-2))^{1/2}\ \cdot\ )$.

{\em Estimate for $\iukw$:}
\begin{lemma}\label{uw}
\be
    \left|
    \iukw
    \right|\leq\left\{
    \begin{array}{ll}
        O\left(\eps_k^\frac 32\nwk\right)&\mbox{if\ }N=5,\\
        O\left(\eps_k^2|\log\ek|^\frac 23\nwk\right)&\mbox{if\ }N=6,\\
        O\left(\eps_k^2\nwk\right)&\mbox{if\ }N\geq 7.
    \end{array}
    \right.
\ee
\end{lemma}
\begin{proof}
    $$\left|\iukw\right|\leq\nwks\left(\int U_{k}^\frac{2N}{N+2}
    \right)^\frac{N+2}{2N}.$$
    If $N=5$, then $\frac{2N}{N+2}<\frac N{N-2}$. By~(\ref{um}),
    $$\left|\iukw\right|\leq C\nwk\eps_{k}^{\frac{N-2}2}=
    O\left(\eps_k^\frac 32\nwk\right).$$
    If $N=6$, then $\frac{2N}{N+2}=\frac N{N-2}$. By~(\ref{dois}),
    $$\left|\iukw\right|\leq C\nwk\left(\eps_{k}^{
    \frac 
    N2}|\log\ek|\right)^\frac{N+2}{2N}=O\left(\eps_{k}^2|\log\ek|^\frac 
    23\nwk\right).$$
    If $N\geq 7$, then $\frac N{N-2}<\frac{2N}{N+2}$. By~(\ref{tres}),
$$\left|\iukw\right|\leq C\nwk\left(\eps_{k}^{N\left(
1-\frac{2N}{N+2}\frac{N-2}{2N}\right)}\right)^\frac{N+2}{2N}=O
\left(\eps_{k}^2\nwk
\right).$$
\end{proof}

{\em Estimate for $\iusu$:\/} From [APY], Equations (3.15), for $N\geq 
5$,
\be\label{iusu}
\iusu=O(\ek\nwk).\ee

{\em Estimate for $\iuzu$:\/} Since $\frac{2N}{N+2}>1$,
\ba\label{iuzu}
\iuzu&\leq&\nwks\left(\int 
U_{k}^{\frac{N}{N-2}\frac{2N}{N+2}}\right)^\frac{N+2}{2N}\nonumber\\
&\leq&C\nwks\eps_{k}^{\left(1-{\frac{N}{N+2}}
\right){\frac{N+2}{2}}}\nonumber\\
&=&O(\ek\nwk).
\ea

{\em Estimate for $\iusd$:\/} 
\be\label{iusd}
\iusd=O(\nwkq).
\ee

Now we will obtain a lower bound for $I_{\ak}(\uk)$.
Let $v_{k}=\uk/\ck=
U_{k}+\tilde{w}_{k}=U_{k}+\wk/\ck$.  Because of (\ref{ck}), 
the sequence $(v_{k})$ 
satisfies (\ref{sndd}) and the sequence $\tilde{w}_{k}$ 
satisfies (\ref{zero}).
Of course, $d(v_{k},M)$ is achieved by $\Uk$.
Because $I$ is 
homogeneous of degree zero, $I_{\ak}(\uk)=I_{\ak}(v_{k})$.  
We will compute 
$I_{\ak}(v_{k})$ but we will still call $v_{k}$ by $\uk$, and 
$\tilde{w}_{k}$
by $\wk$. 

Going back to (\ref{lb}), $I_{\alpha}(u_{k})$ is bounded below 
by the sum of 
$\beta(u_{k})$ and $\frac{4}{\tz}\beta(u_{k})\delta(u_{k})$.  
We start by obtaining lower bounds 
for $\beta(u_{k})$ and $\frac{4}{\tz}\beta(u_{k})\delta(u_{k})$ 
separately.
The expression for $\beta(u_{k})$ involves two terms: $\nhukpq$ 
and $\nukpss$.
The first one is obviously
\begin{eqnarray}
    \nhukpq&=&\pnhukq+2{\txt\left(\igukw+a\iukw\right)} 
    +\ \pnhwkq\label{A}\\
    &=&A_{1}+A_{2}+A_{3}.\nonumber
\end{eqnarray}

For the second term we use
\begin{lemma}[{[APY] Lemma 3.5}]\label{expansion} Let $q>1$ and $L$ be a 
non negative integer with $L\leq q$.  Let $V$ and $\omega$ be 
measurable functions on $\Om$ with $V\geq 0$ and $V+\omega\geq 0$.
Then
\bas
\int (V+\omega)^q&=&\sum_{i=0}^{L}\frac{q(q-1)\ldots(q-i+1)}{i!}\int 
V^{q-i}\omega^i\\
&&+\ O\left(\int[V^{q-r}|\omega|^r+|\omega|^q]\right),
\eas
where $r=\min\{L+1,q\}$.
\end{lemma}
Taking $L=2$ and $q=\ts$,
\be\label{BB}\!\!\!\!\!\!\!\!\!\!\!\!
\nukpss\!=\iuks+\ts\iusu+{\txt\frac{\ts(\tsu)}{2}}\iusd+O(\nwkr),
\ee
where $r=\min\{\ts,3\}$, i.e., $r=3$ if $N=5$, and $r=\ts$ if $N>5$.
The inequality
\be\label{ineq}
(1+z)^{-\eta}\geq 1-\eta z,
\ee
for $\eta>0$ and $z\geq -1$, implies
\begin{eqnarray}
    |u_{k}|_{\ts}^{-2}&\geq&\nuksmd
    \left(1-\frac{2\iusu}{\nukss}
    \right.\nonumber\\ &&\qquad\qquad\ \ 
    \left.-
    \frac{{(\tsu)}\iusd}{\nukss}+O(\nwkr)\right)\label{B}\\
    &=&B_{1}+B_{2}+B_{3}+B_{4}.\nonumber
\end{eqnarray}

Let
$$l:=\ql.$$
From (\ref{ngukq}) and (\ref{iuks}),
\be\label{ql}
l=1+O(\ek).
\ee

Using (\ref{A}) and (\ref{B}),
we can write,  
$$\beta(u_{k})\geq \bar I_{1}+\bar I_{2}+\bar I_{3}+\bar I_{4},$$
where
\begin{eqnarray*}
    \bar I_{1}&=&\frac{\nhukq}{\nuksq} 
    \\
&=&A_{1}B_{1} 
\\
    \bar I_{2}&=&\frac{2}{\nuksq}\left[
    \igukw+a\iukw-l\iusu
    \right]\\
    \\
    &=&A_{2}B_{1}+A_{1}B_{2} 
    \\
    \bar I_{3}&=&\frac{1}{\nuksq}\left[
    \nhwkq-l(\tsu)\iusd
    \right]\\
    \\
    &=&A_{3}B_{1}+A_{1}B_{3} 
    \\
\noalign{\mbox{and}}
\bar I_{4}&=&[(A_{1}+A_{3})B_{4}]+[A_{2}(B_{2}+B_{3}+B_{4})]+A_{3}B_{2}+
A_{3}B_{3}\\
&=&E_{1}+E_{2}+E_{3}+E_{4}.
\end{eqnarray*}

By (\ref{nukq}) and (\ref{iuks}),
$$
 \bar I_{1}=\frac{\ngukq}{\nuksq}+
o(\ek).
$$

We recall (\ref{zero}), $\wk\to 0$ in $H^1(\Om)$.

By (\ref{dot}), the first of the four terms in $\bar I_{2}$ is zero;
by Lemma~\ref{uw} and by (\ref{iusu}) the second and the third ones
are $o(\ek)$: 
$$
\bar I_{2}=o(\ek).
$$

By (\ref{iuks}), (\ref{iusd}) and (\ref{ql}),
$$
  \bar I_{3}=2^{\frac{N-2}N}S^{\frac{2-N}{2}}\left[
    \nhwkq-(\tsu)\iusd\right]
$$

The term $E_{1}$ is $o(\nwkq)$ because $B_{4}$ is $o(\nwkq)$.
The term $E_{2}$ is $o(\ek)$ because, from Lemma~\ref{uw},
$A_{2}$ is $o(\ek)$.
The term $E_{3}$ is $o(\ek)$ because, from (\ref{iusu}),
$B_{2}$ is $o(\ek)$.
Finally, the term $E_{4}$ is $o(\nwkq)$ because 
both $A_{3}$ and $B_{3}$ are 
$O(\nwkq)$.
Therefore, $$\bar I_{4}=o(\ek)+o(\nwkq).$$

Combining the expressions for $\bar I_{1}$, $\bar I_{2}$,
$\bar I_{3}$ and $\bar I_{4}$,
\bas \beta(u_{k})
    &=& 
    \frac{\ngukq}{\nuksq}+
    2^{\frac{N-2}N}S^{\frac{2-N}{2}}\left[
    \pnhwkq
        -(\tsu)\iusd
    \right]\\ &&+\ o(\ek)+o\left(\nwkq\right)\\
    &\geq&\frac{\ngukq}{\nuksq}+
    2^{\frac{N-2}N}S^{\frac{2-N}{2}}\left[
    \gamma_{2}\pnhwkq
        -(\tsu)\iusd
    \right]+o(\ek),
\eas
for any fixed number $\gamma_{2}<1$, because $a>0$. 
This is our lower bound for $\beta(u_{k})$.

Now we turn to the term
$\frac 4\tz\beta(u_{k})\delta(u_{k})$ and write
\be\label{prod}\frac 4\tz\beta(u_{k})\delta(u_{k})=\frac 2\tz
\frac{\pnhukpqmeio}{|u_{k}|_{\ts}^{2+\ts\!/2}}\ak\iukpz
\ee

We obtain a lower bound for $\pnhukpqmeio$ from (\ref{A}).  Using
(\ref{dot}), (\ref{nukq}), (\ref{ngukq}) and Lemma~\ref{uw},
$$
\pnhukpqmeio\geq \left(\sndd\right)^\frac{1}{2}+O(\ek)+O(\nwkq).
$$

We obtain a lower bound for $|u_{k}|_{\ts}^{-(2+\ts\!/2)}$ from 
(\ref{BB}). Using  (\ref{iuks}), (\ref{iusu}), (\ref{iusd}) and 
(\ref{ineq}),
$$
|u_{k}|_{\ts}^{-(2+\ts\!/2)}
\geq\left(\sndd\right)^{-\frac{1}{2}-\frac{2}{\ts}}+
O(\ek)+O(\nwkq).
$$

For the product we obtain the lower bound
\ba
\frac{\pnhukpqmeio}{|u_{k}|_{\ts}^{2+\ts\!/2}}&\geq&
2^{\frac{N-2}{N}}S^{\frac{2-N}{2}}+O(\ek)+O(\nwkq)\label{fum}\\
&=&D_{1}+D_{2}+D_{3}.\nonumber
\ea

To estimate the term $\ak\iukpz$ we do not use Lemma~\ref{expansion}
because it would 
give rise to a term $O\left(\ak\nwk^\tz\right)$, for 
which we do not have estimates.  Instead we use this calculus
\begin{lemma}\label{calculus}
    Let $\eta>2$.  For any $z\geq-1$,
    \be\label{calc}
    \frac{\eta(\eta-1)}2 z^2-\ct|z|+1\leq(z+1)^\eta,
    \ee
    where $\ct=1+\eta(\eta-1)/2$.
\end{lemma}
\begin{proof}
The difference between the right hand side and the left hand side is 
zero for $z=-1$ and $z=0$.  It is increasing for $z>0$ and concave for
$-1<z<0$.
\end{proof}
((\ref{calc}) also hold for $\eta=2$, with equality for negative values 
of $z$.)

As a consequence of Lemma~\ref{calculus},
$$
    \iukpz\geq\iukz-\tz\cc\iuzu+{\txt\frac{\tz(\tzu)}{2}}\iuzd,
$$
with
$$
    \hat{C}:=\frac{\tilde{C}}{\tz}=\frac{1}{\tz}+\frac{\tzu}{2}.
$$
Using (\ref{iukz}) and (\ref{iuzu}),
\ba\label{F}
\frac{2}{\tz}\ak\iukpz&\geq&B(N)\ak\ek+(\tzu)\ak\iuzd+o(\ak\ek)\\
&=&F_{1}+F_{2}+F_{3}.\nonumber
\ea

We will now substitute (\ref{fum}) and (\ref{F}) in (\ref{prod}).
On the one hand, $$(D_{1}+D_{2}+D_{3})F_{3}=o(\ak\ek)$$ 
and
$$
(D_{2}+D_{3})F_{1}=o(\ak\ek).
$$
On the other hand, by (\ref{iuzd}),
$$
D_{2}F_{2}=O\left(\ak\eps_{k}^2|\log\ek|^{\frac{2}{N}}
\nwkq\right)=o(\ak\ek).
$$
So,
\bas
\frac 4\tz\beta(u_{k})\delta(u_{k})&\geq&
2^{\frac{N-2}{N}}S^{\frac{2-N}{2}}B(N)\ak\ek\\  &&
+\ 2^{\frac{N-2}{N}}S^{\frac{2-N}{2}}(\tzu)\ak\iuzd\\
&&+\ O\left(\nwkq\right)\ak\iuzd+o(\ak\ek)\\
&\geq&
2^{\frac{N-2}{N}}S^{\frac{2-N}{2}}\left[B(N)\ak\ek+
\gamma_{2}(\tzu)\ak\iuzd\right]\\
&&+\ o(\ak\ek),
\eas
for any fixed number $\gamma_{2}<1$. 
This is our lower bound for $\frac 4\tz\beta(u_{k})\delta(u_{k})$.

Combining the lower bounds for $\beta(u_{k})$ and for 
$\frac 4\tz\beta(u_{k})\delta(u_{k})$,
\bas
    I_{\ak}(\uk)&\geq&
    \frac{\ngukq}{\nuksq}+2^{\frac{N-2}{N}}S^{\frac{2-N}{2}}B(N)\ak\ek\\
    &&+\ 
    2^{\frac{N-2}N}S^{\frac{2-N}{2}}\left[\gamma_{2}
    \pnhwkq+\gamma_{2}(\tzu)\ak\iuzd
    \right.\\
    &&\left.\qquad -\ (\tsu)\iusd\right]+o(\ak\ek).
\eas
From Lemma~\ref{dif}, the term inside the square parenthesis is greater 
than
$$
    \left[\left(
    \gamma_{2}-\frac{(\tsu)}{(\tsu)+\gamma_{1}}\right)\left(\pnhwkq
    +(\tzu)\ak\iuzd\right)+o(\eps_{k})
    \right].
$$
Choosing $\gamma_{2}\geq\frac{(\tsu)}{(\tsu)+\gamma_{1}}$, yields
that this term is greater than $o(\eps_{k})$.  Hence,
$$
I_{\ak}(\uk)\geq\frac{\ngukq}{\nuksq}+2^{\frac{N-2}N}S^\frac{2-N}{2}
B(N)\ak\ek+o(\ak\ek).
$$
Substituting (\ref{ei}) into this expression, we obtain
\bas
I_{\ak}(\uk)&\geq&
\frac{S}{2^\frac 2N}
+2^{\frac{N-2}N}S^\frac{2-N}{2}B(N)\ak\ek\left[1-
S^{\frac N2}\frac{A(N)}{B(N)}H(y_{k})\frac 1\ak+o(1)
\right]\\
&>&\frac{S}{2^\frac 2N},
\eas
for large $k$.

So assume $\alpha_{0}$, in (\ref{min}),  is $+\infty$.  Choose a
    sequence $\alpha_{k}\to+\infty$ as $k\to+\infty$ and denote by $\uk$ 
    a minimizer for $I_{\alpha_{k}}$ satisfying \rfk.  From 
    Lemmas~\ref{infinito} and \ref{instantao}, the conditions 
    (\ref{sndd}), (\ref{m}), 
    (\ref{ck}), (\ref{akek}) and (\ref{zero}) hold.  Therefore 
    $S_{\ak}=I_{\ak}(\uk)>\sddn$ for large $k$, which is impossible.
    By Corollary~\ref{coro}, this 
    establishes (ii) of Theorem~\ref{theorem}. 
    
\begin{remark}
    Since $S^\frac{N}{2}=\int_{\Rb^{N}} U^\ts=\omega_{N}
    \frac{1}{2^{N}}\sqrt{\pi}
    \frac{\Gamma\left(\frac{N}{2}\right)}{\Gamma\left(\frac{N+1}{2}\right)}
    [N(N-2)]^\frac{N}{2}$, it follows that
    $$B(N)=S^{\frac N2}.$$
    Using 
    $$
    \omega_{N}=\frac{2\pi^\frac{N}{2}}{\Gamma\lb
    \frac{N}{2}\rb},$$ the common 
    value is
    $$B(N)=S^{\frac N2}=\frac{\pi^\frac{N+1}{2}}{2^{N-1}}
    \frac{1}{\Gamma\lb\frac{N+1}{2}\rb}[N(N-2)]^\frac{N}{2}.
    $$
\end{remark}
    
\section{Least energy solutions of \rfz}

    In this section we give a lower bound for 
    $\alpha_{0}=\min\left\{\alpha\,|\;S_{\alpha}=S/2^\frac{2}{N}\right\}$,
    and give partial results concerning existence of 
    least energy solutions of \rfz.
        
    From (\ref{iau}) we obtain 
\begin{lemma}\label{lemmau}
    There exists a constant 
    $\bar{c}>
    \frac{4}{(\tz)^{{2}/{N}}}$ such that
    \be\label{up}I_{\alpha}\leq\beta\left(1+\frac{4}{\tz}
    \delta+\bar{c}\delta^2\right).\ee
\end{lemma}
\begin{proof}
    Consider $\Lambda:[0,+\infty[\to\Rb$, defined by 
    $$\Lambda(\bd):=\frac{1}{\left(\tz\right)^{\frac{2}{N}}}
     \left[\left(\bd+\sqrt{\bd^2+1}\right)^{N}+
         \frac\ts{2\;}\left(\bd+\sqrt{\bd^2+1}\right)^{N-2}
         \right]^\frac 2N.$$
         Since $\left.\frac{\partial}{\partial\bd}\sqrt{\bd^2+1}
         \right|_{\bd=0}=0$ and 
         $\left.\frac{\partial}{\partial\bd}\frac{1}{\sqrt{\bd^2+1}}
         \right|_{\bd=0}=0$,
         the first two derivatives of $\Lambda$ at zero are
         $$
         \Lambda'(0)=\frac{1}{\left(\tz\right)^\frac 2N}
         \frac{2}{N}(\tz)^{\frac{2}{N}-1}\left[N+\frac{\ts}{2}(N-2)\right]=
         \frac{4}{\tz}
         $$
         and
         $$
         \Lambda^{\prime\prime}(0)=
         {\textstyle\frac{1}{\left(\tz\right)^\frac{2}{N}}\frac{2}{N}
         \left(\frac{2}{N}-1\right)(\tz)^{\frac{2}{N}-2}(2N)^2+\frac{2}{\tz}
         [N+(N-2)]=\frac{4}{\tz}\frac{2N-3}{N-1}}.
         $$
         Fix any number $c_{1}>\frac{2}{\tz}\frac{2N-3}{N-1}$.  
         There exists an $\epsilon>0$ such that 
         (\ref{up}) holds for $\bar{c}=c_{1}$ and $0\leq\bd<\epsilon$.
         
         Fix any number $c_{2}>\frac{4}{(\tz)^{{2}/{N}}}$.  From 
         (\ref{iau}), there exists 
         an $L>0$ such that (\ref{up}) holds for $\bar{c}=c_{2}$ and
         $\bd>L$.
         
         The inequalities $\frac{2}{\tz}\frac{2N-3}{N-1}<
         \frac{4}{\tz}<\frac{4}{(\tz)^{{2}/{N}}}$ show that
         $\max\left\{\frac{2}{\tz}\frac{2N-3}{N-1},\right.$
         $\left.\frac{4}{(\tz)^{{2}/{N}}}
         \right\}=\frac{4}{(\tz)^{{2}/{N}}}$.
         
         By taking $\bar{c}\geq\max\{c_{1},c_{2}\}$, $\bar{c}$ 
         sufficiently large,
         we can guarantee (\ref{up}) for all $\bd\in[\epsilon,L]$.
\end{proof}
\begin{lemma}
    If 
    $\alpha<{A(N)}\max_{\partial\Om}H$, 
    then $S_{\alpha}<\sddn$.
\end{lemma}
\begin{proof}
    Choose $P\in\partial\Om$ such that $H(P)=\max_{\partial\Om}H$. 
    From (\ref{nukq}) and (\ref{ei}),
    $$\beta(U_{\eps,P})=\frac{S}{2^\frac 
2N}-2^{\frac{N-2}N}SH(P)
A(N)\eps+o(\eps),$$ 
whereas, from (\ref{delta}) and (\ref{nukq})-(\ref{iuks}),
$$
\delta(U_{\eps,P})=\frac{2}{S^\frac{N}{2}}
\frac{\tz}{4}B(N)\alpha\eps+o(\eps).
$$
The previous lemma implies that
\bas
S_{\alpha}&\leq&I_{\alpha}(U_{\eps,P})\\ &\leq&
\frac{S}{2^\frac 2N}
-2^{\frac{N-2}N}S^\frac{2-N}{2}B(N)\alpha\eps\left[
S^{\frac N2}\frac{A(N)}{B(N)}H(P)\frac 1\alpha-1+o(1)
\right]\\
&=&
\frac{S}{2^\frac 2N}
-2^{\frac{N-2}N}S\alpha\eps\left[
{A(N)}H(P)\frac 1\alpha-1+o(1)
\right]
\eas
as $\eps\to 0$.  Since, by assumption, 
$\alpha<A(N)\max_{\partial\Om}H=A(N)H(P)$, $S_{\alpha}<\sddn$.
\end{proof}
\begin{corollary}\label{asndd}
    $\!\!$The value $\alpha_{0}$ is greater than or equal to 
    ${A(N)}\max_{\partial\Om}H$.
\end{corollary}    

We let $|\Om|$ denote the Lebesgue measure of $\Om$.
By testing $I_{\alpha}$ with constant functions we obtain
\begin{lemma}\label{amax} 
    $\!\!$If\ $a\leq\frac{S}{(2|\Om|)^\frac{2}{N}}$, then $\alpha_{0}
    \geq\max\left\{
    \alpha\in
    [0,+\infty[\,\left|I_{\alpha}(1)\leq\sddn\right.\right\}$.
\end{lemma}
{\em Note.} The value of $I_{\alpha}(1)$ is
$$
\textstyle
I_{\alpha}(1)=\frac{|\Om|^\frac{2}{N}}{(\tz)^\frac{2}{N}}
\left[
\lb\frac{\alpha+\sqrt{\alpha^2+4a}}{2}\rb^{N}+
\frac{\ts}{2\,}a\lb\frac{\alpha+\sqrt{\alpha^2+4a}}{2}\rb^{N-2}
\right]^\frac{2}{N}.$$
    
    {\em We have not determined the exact value of $\alpha_{0}$.} 
    However, using the ideas of Chabrowski and Willem~\cite{CW},
    we have the following proposition concerning existence of 
    least energy solutions for $\alpha=\alpha_{0}$:
\begin{proposition}\label{aaz}
    If $\alpha_{0}>\azz$ then there exists a least energy solution of 
    \rfz.
\end{proposition}
\begin{proof}
    Choose a sequence 
    $\ak\nearrow\alpha_{0}$.  Let $\uk$ be a minimizer of 
    $I_{\alpha_{k}}$ satisfying \rfk. As in the proof of 
    Lemma~\ref{infinito}, we conclude that the sequence $(\uk)$ is 
    bounded in $H^1(\Om)$.  So we can assume $\uk\weak u$.
    
    We claim that $u\neq 0$.  Suppose, by contradiction, that $u=0$.
    If the norms $|u_{k}|_{L^\infty(\Om)}$ are uniformly bounded, then,
    from (\ref{sndd}), $\ius=\sndd$, which contradicts $u=0$.
    If $|u_{k}|_{L^\infty(\Om)}\to+\infty$, then Lemma~\ref{instantao} 
    implies that we can repeat the argument of the previous sections 
    to conclude that $S_{\ak}>\sddn$, for large $k$.  This is also a 
    contradiction.  So $u\neq 0$.
    
    Since $u\neq 0$, the argument in the proof of 
    Lemmas~\ref{pscondition} and \ref{cont}
    yields that $u$ is a least energy solution of \rfz.
    Indeed, with the notations in the proof of Lemma~\ref{pscondition},
    $x_{0}\neq 0$.  If 
    $\left[h(1)\right]^{\frac{2}{N}}/[{4(\tz)^{\frac{2}{N}}}]
    >\sddn$, then $S_{\alpha_{0}}=
    \left[h(x_{0})\right]^{\frac{2}{N}}/[{4(\tz)^{\frac{2}{N}}}] 
    >\sddn$.  Hence 
    $I_{\alpha_{0}}(u)=
    \left[h(1)\right]^{\frac{2}{N}}/[{4(\tz)^{\frac{2}{N}}}]
    =\sddn$.
\end{proof}

\begin{remark}
    If $a$ is sufficiently small and $\Om$ is a (unit) ball, then
    the lower bound for $\alpha_{0}$ in {\rm Corollary~\ref{asndd}} is 
    smaller 
    than the lower bound for $\alpha_{0}$ in\/ {\rm Lemma~\ref{amax}} 
    so that the previous proposition applies. 
\end{remark}
\begin{proof}
    The lower bound for $\alpha_{0}$ in Corollary \ref{asndd} is
    $A(N)$, given in (\ref{aden}).  As $a\to 0$, the lower bound for 
    $\alpha_{0}$ in Lemma \ref{amax} tends to
    \bas\lb\frac{(\tz)^\frac{2}{N}}{|\Om|^\frac{2}{N}}
    \sddn\rb^\frac{1}{2}&=&
    \lb\frac{\tz}{2}\rb^\frac{1}{N}S^\frac{1}{2}
    \frac{1}{|\Om|^\frac{1}{N}}\\&=&
    {\textstyle \lb\frac{N-1}{N-2}\rb^\frac{1}{N}} 
     \frac{\pi^\frac{N+1}{2N}}{2^{\frac{N-1}{N}}}
    \frac{1}{\left[\Gamma
    \lb\frac{N+1}{2}\rb\right]^\frac{1}{N}}[N(N-2)]^\frac{1}{2} 
    \frac{\left[\Gamma
    \lb\frac{N+2}{2}\rb\right]^\frac{1}{N}}{\pi^\frac{1}{2}}
    \\&=&
    \left[\frac{\pi^\frac{1}{2}}{2^{N-1}}
    {\textstyle \lb\frac{N-1}{N-2}\rb}\frac{\Gamma
    \lb\frac{N+2}{2}\rb}{\Gamma\lb\frac{N+1}{2}
    \rb}\right]^\frac{1}{N}[N(N-2)]^\frac{1}{2}.
    \eas
\end{proof}

    Suppose now $\alpha_{0}=\azz$.  Once again, choose a sequence 
    $\ak\nearrow\alpha_{0}$ and let $\uk$ be a minimizer of 
    $I_{\alpha_{k}}$ satisfying \rfk.  The argument in the proof of 
    the previous proposition shows that, modulo a subsequence, either 
    $\uk\weak u\neq 0$, or $\uk\weak 0$ and 
    $|u_{k}|_{L^\infty(\Om)}\to+\infty$.
    {\em We have not determined which of these alternatives holds.}
    In the first case $u$ is a least energy solution of \rfz.  In the 
    second case let, as before, $P_{k}$ be such that 
    $\uk(P_{k})=|u_{k}|_{L^\infty(\Om)}$.  Any limit point of 
    $(P_{k})$ is contained in the set of points of maximum 
    mean curvature of $\partial\Om$.  For if $y_{0}$ is a limit point 
    of $P_{k}$, then
    \bas -2^{\frac{N-2}N}SH(y_{k})A(N)\ek&=&\left[
    2^{\frac{N-2}N}SH(y_{0})A(N)\ek
    -2^{\frac{N-2}N}SH(y_{k})A(N)\ek\right]\\
    &&-
    2^{\frac{N-2}N}SH(y_{0})A(N)\ek\\  
    &=&
    -2^{\frac{N-2}N}SH(y_{0})A(N)\ek+o(\ek).
    \eas
    If $H(y_{0})<\max_{\partial\Om}H$, then the argument in the 
    previous section shows that $S_{\ak}>\sddn$, for large $k$.
   
    We summarize these observations in
\begin{proposition}
    Suppose $\alpha_{0}=\azz$. Then
    \begin{enumerate}
    \item[(i)]\ either there exists a least 
    energy solution of \rfz,
    \item[(ii)]\ or any sequence, $\uk$, of least 
    energy solutions of \rfk, for $\ak<\alpha_{0}$, 
    $\ak\to\alpha_{0}$, has a subsequence,
    $\uk$, $\uk\weak 0$, 
    $|\uk|_{L^\infty(\Om)}\to+\infty$; the limit points of any sequence 
    of maximums of $\uk$ are contained in the set of points of maximum 
    mean curvature of the boundary of $\Om$.  
    \end{enumerate}
\end{proposition}    
    
\appendix		
    
\section{The functional restricted to the Nehari manifold}

In this Appendix we start by checking, using standard arguments, 
that the Nehari set 
${\cal{N}}$ is a manifold and a natural constraint for 
$\Phi_{\alpha}$ (defined in (\ref{phi})).  We then derive the 
expressions (\ref{psiu}) and (\ref{psid}) for $\Phi_{\alpha}$ 
restricted to ${\cal{N}}$, we derive an expression for $I_{\alpha}$ 
(defined in
(\ref{defi})) equivalent to 
(\ref{energiad}) and to (\ref{iau}), and we derive 
upper and lower bounds for $I_{\alpha}$.

Consider the set
$${\cal{N}}:=\left\{u\in H^1(\Om):
\Phi_{\alpha}^\prime(u)u=0, u\neq 0\right\},$$
where $\Phi_{\alpha}$ is the $C^2$ functional defined in (\ref{phi}),
and define $J_{\alpha}:H^1(\Om)\to\Rb$ by
$$
J_{\alpha}(u):=\Phi_{\alpha}^\prime(u)u=\nhu+\alpha\iuz-\ius.
$$
The set ${\cal{N}}=\left\{u\in H^1(\Om):J_{\alpha}(u)=0, u\neq 0\right\},$
is a manifold (called the Nehari manifold).  
Indeed, if $u\in {\cal{N}}$, then
$J_{\alpha}^\prime(u)\neq 0$, because if $J_{\alpha}(u)=0$ and
$J_{\alpha}^\prime(u)u=0$, then
$$
0=\ts J_{\alpha}(u)-J_{\alpha}^\prime(u)u=(\ts-2)\pnhu+(\ts-\tz)\alpha\iuz.
$$
This yields $u=0$.
Furthermore, the Nehari manifold is a natural constraint for
$\Phi_{\alpha}$, by which we mean that any critical point of 
$\Phi_{\alpha}|_{{\cal{N}}}$ 
is a critical point of $\Phi_{\alpha}$.  In fact, suppose that 
$u\in {\cal{N}}$ is a critical point of $\Phi_{\alpha}|_{{\cal{N}}}$.  
Then there 
exists a $\lambda\in\Rb$ such that
$\Phi_{\alpha}^\prime(u)=\lambda J_{\alpha}^\prime(u)$.  Applying both 
sides to $u$,
$0=J_{\alpha}(u)=\Phi_{\alpha}^\prime(u)u=\lambda J_{\alpha}^\prime(u)u$. 
However, we just 
saw that $J_{\alpha}^\prime(u)u\neq 0$ if $J_{\alpha}(u)=0$ (and $u\neq 0$).  
It follows that
$\lambda=0$ and $u$ is a critical point of $\Phi_{\alpha}$.

For any $u\in H^1(\Om)\setminus\{0\}$ there exists a unique $t(u)>0$ 
such that $t(u)u\in {\cal{N}}$, i.e.\ $\Phi_{\alpha}^\prime(t(u)u)t(u)u=0$.  
The value of $t(u)$ is the solution of
$$
\nhu+\alpha\iuz [t(u)]^{\tz-2}-\ius [t(u)]^{\ts-2}=0.
$$
Since $\tz-2=\frac 2{N-2}$ is half of $\ts-2$, the equation
$$
\aaa+bt^{\tz-2}-ct^{\ts-2}=0.
$$
is 
quadratic in $t^\frac 2{N-2}$. 
Define the functionals $\aaa$, $b$ and $c:H^1(\Om)\setminus\{0\}\to
\Rb$ by
$$\begin{array}{rrlll}
\aaa(u)&:=&\nhu,&&\\
b(u)&:=&\alpha\iuz&=&b_{\alpha}(u),\\
c(u)&:=&\ius.&&
\end{array}$$
(Note that $\aaa\neq a$.)
The value of $t(u)$ is
\be\label{t}
t(u)=
\left(\frac{b+\sqrt{b^2+4\aaa c}}{2c}\right)^\frac{N-2}{2}(u).
\ee
The functional $t:H^1(\Om)\setminus\{0\}\to\Rb$ is obviously 
continuous and the map $u\mapsto t(u)u$ defines a homeomorphism of the 
unit sphere in $H^1(\Om)$ with ${\cal{N}}$.  Its inverse is the retraction
$u\mapsto\frac u{||u||}$.

We define $\Psi_{\alpha}:H^1(\Om)\setminus\{0\}\to\Rb$ by
$$\Psi_{\alpha}(u):=\Phi_{\alpha}(t(u)u).$$
In terms of $\aaa$, $b$, $c$ and $t$, 
$$
\Psi_{\alpha}=\frac 12\aaa t^2+\frac 1\tz bt^\tz-\frac 1\ts ct^\ts.
$$
Replacing (\ref{t}) into this expression for $\Psi_{\alpha}$, and 
simplifying, leads to 
$$
\Psi_{\alpha}=\frac 1N\frac 1\tz\left[
\left(\frac{b+\sqrt{b^2+4\aaa c}}{2c}\right)^Nc+\frac{\ts}{2\;}
\left(\frac{b+\sqrt{b^2+4\aaa c}}{2c}\right)^{N-2}\aaa
\right].
$$
We now introduce the functionals $\beta$, 
$\gamma:H^1(\Om)\setminus\{0\}\to\Rb$, defined by
$$
\beta:=\frac \aaa{c^\frac{N-2}N}
$$
and
$$
\gamma=\gamma_{\alpha}:=\frac b{c^\frac{N-1}N},
$$
as in expressions (\ref{defb}) and (\ref{defg}), respectively.  In terms of 
$\beta$ and $\gamma$, the expression for $\Psi_{\alpha}$ is
$$\Psi_{\alpha}=\frac 1N\frac 1\tz\frac 1{2^{N}}
         \left[\left(\gb\right)^{N}+2\cdot\ts\beta\left(\gb\right)^{N-2}
         \right].
$$
This is (\ref{psiu}).  If we introduce still another functional 
$\delta:H^1(\Om)\setminus\{0\}\to\Rb$, defined by
$$
\delta=\delta_{\alpha}:=\frac{\gamma}{2\sqrt\beta},
$$
as in expression (\ref{delta}), then we can write $\Psi_{\alpha}$ as
$$
\Psi_{\alpha}=\frac 1N\frac{\beta^\frac N2}\tz
         \left[\left(\delta+\sqrt{\delta^2+1}\right)^{N}+
         \frac\ts{2\;}\left(\delta+\sqrt{\delta^2+1}\right)^{N-2}
         \right].
$$
This is (\ref{psid}).

We give an expression for $I_{\alpha}=
\left(N\Psi_{\alpha}\right)^\frac 2N$, defined in
(\ref{defi}), equivalent to  
(\ref{energiad}) and to (\ref{iau}):
$$
I_{\alpha}=\beta\left(
\delta+\sqrt{\delta^2+1}
\right)^\frac 4\ts
\left(
\frac 2\tz\delta^2+\frac 2\tz\delta\sqrt{\delta^2+1}+1
\right)^\frac 2N.
$$
Since 
$$
\frac 4\ts+\frac 2N\frac 2\tz=\frac 4\tz,
$$
$I_{\alpha}$ has the lower bound
\be\label{lbd}
I_{\alpha}\geq \beta\left(1+\frac 4\tz\delta\right).
\ee
For an upper bound for $I_{\alpha}$ we refer to Lemma~\ref{lemmau}.

\section{The estimate for $\iukz$}

In this Appendix we use the ideas of Adimurthi and Mancini~\cite{AM} 
to prove (\ref{iukz}).

We wish to estimate $|U_{\eps,y}|_{\tz}^\tz$, where 
$U_{\eps,y}$ is defined in (\ref{resins}) and $y\in\partial\Omega$.
By a change of coordinates we can assume that $y=0$,
$$
B_{R}(0)\cap\Om=\{(x',x_{N})\in B_{R}(0)|x_{N}>\rho(x')\}
$$
and
$$
B_{R}(0)\cap\partial\Om=\{(x',x_{N})\in B_{R}(0)|x_{N}=\rho(x')\},
$$
for some $R>0$, where $x'=(x_{1},\ldots,x_{N-1})$,
$$\rho(x')=\sum_{i=1}^{N-1}\lambda_{i}x_{i}^2+O(|x'|^3),$$
$\lambda_{i}\in\Rb$, $1\leq i\leq N-1$.  

We begin by supposing 
all the $\lambda_{i}$'s are positive.
Let 
$U_{\eps}:=U_{\eps,0}$ and
$\Sigma:=\{(x',x_{N})\in B_{R}(0)|0<x_{N}<\rho(x')\}$.
Then
\be\label{sumudt}
|U_{\eps}|_{\tz}^\tz=
\frac 12\int_{B_{R}(0)}U_{\eps}^\tz-
\int_{\Sigma}U_{\eps}^\tz+\int_{B_{R}^{C}(0)\cap\Om}U_{\eps}^\tz.
\ee

We will estimate each of the three terms on the right hand side of
(\ref{sumudt}).  For the third term we have
\bas
\int_{B_{R}^{C}(0)\cap\Om}U_{\eps}^\tz&\leq&
\int_{B_{R}^{C}(0)}U_{\eps}^\tz\\
&=&O\left(\eps\int_{R/\eps}^{+\infty}
\frac{r^{N-1}}{(1+r^2)^{N-1}}dr\right)\\
&=&O(\eps\times\eps^{N-2})\\
&=&O(\eps^{N-1})
\eas

Using this estimate, for the first term on the 
right hand side of (\ref{sumudt}) 
we have
\bas
\frac 12\int_{B_{R}(0)}U_{\eps}^\tz&=&\frac 
12\int_{\Rb^{N}}U_{\eps}^\tz+O(\eps^{N-1})\\
&=&\frac{1}{2}\eps\int_{\Rb^{N}}U^\tz+O(\eps^{N-1})\\
&=&\frac{\tz}{2}B(N)\eps +O(\eps^{N-1}),
\eas
with
\bas
B(N)&:=&\frac 1\tz\int_{\Rb^{N}}U^\tz\\
&=&\frac 1\tz\,\omega_{N}\int_{0}^{+\infty}
\frac{r^{N-1}}{(1+r^2)^{N-1}}\,dr\times[N(N-2)]^\frac{N}{2}\\
&=&{\textstyle \frac{N-2}{2(N-1)}}\,
\omega_{N}\times
\frac{1}{2^{N-1}}\sqrt{\pi}\,
\frac{\Gamma\left(\frac{N-2}{2}\right)}{\Gamma\left(\frac{N-1}{2}\right)}
\times[N(N-2)]^\frac{N}{2}\\
&=&\omega_{N}\frac{1}{2^{N}}\sqrt{\pi}\,
\frac{\Gamma\left(\frac{N}{2}\right)}{\Gamma\left(\frac{N+1}{2}\right)}
[N(N-2)]^\frac{N}{2};
\eas
here $\omega_{N}$
is the volume of the $N-1$ dimensional unit sphere.

So we are left with the estimate of the second term on the right hand 
side of (\ref{sumudt}).  Let $\sigma>0$ be such that
$$
L_{\sigma}:=\{x\in\Rb^{N}|\,|x_{i}|<\sigma, 1\leq i\leq N\}\subset 
B_{\frac{R}{4}}(0)
$$
and define
$$
\Delta_{\sigma}:=\{x'|\,|x_{i}|<\sigma,1\leq i\leq N-1\}.
$$
For the second term on the right hand side of (\ref{sumudt}),
\bas
   \int_{\Sigma}U_{\eps}^\tz&=&\int_{\Sigma\cap 
   L_{\sigma}}U_{\eps}^\tz+O(\eps^{N-1})\\
   &=&\int_{\Delta_{\sigma}}\int_{0}^{\rho(x')}U_{\eps}^\tz
   dx_{N}\,dx'+O(\eps^{N-1})\\
   &=&O\left(\int_{\Delta_{\sigma}}\int_{0}^{\rho(x')}
   \frac{\eps^{N-1}}{(\eps^2+|x|^2)^{N-1}}\,
   dx_{N}\,dx'\right)+O(\eps^{N-1});\\
\noalign{\noindent using the change of variables 
$\sqrt{\eps^2+|x'|^2}\,y_{N}=x_{N}$,}
   &=&O\left(\int_{\Delta_{\sigma}}
   \frac{\eps^{N-1}}{(\eps^2+|x'|^2)^{N-{\frac{3}{2}}}}   
   \int_{0}^{\frac{\rho(x')}{\sqrt{\eps^2+|x'|^2}}}
   \frac{1}{(1+y_{N}^2)^{N-1}}\,
   dy_{N}\,dx'\right)\\
   &&+\ O(\eps^{N-1});\\
\noalign{\noindent since $\int_{0}^s\frac{1}{(1+t^2)^{N-1}}\,dt=s+O(s^3)$,}
    &=&O\left(\eps^{N-1}\int_{\Delta_{\sigma}}
   \frac{{\textstyle\sum\lambda_{i}x_{i}^2}}{(\eps^2+|x'|^2)^{N-1}}  
   \,dx'\right)\\
   &&+\ O\left(\eps^{N-1}\int_{\Delta_{\sigma}}
   \frac{|x'|^3}{(\eps^2+|x'|^2)^{N-1}}  
   \,dx'\right)\\
   &&+\ O(\eps^{N-1})\\
   &=&O\left(\eps^{2}\int_{\Delta_{\sigma}/\eps}
   \frac{|y'|^2}{(1+|y'|^2)^{N-1}}  
   \,dy'\right)\\
   &&+\ O\left(\eps^{3}\int_{\Delta_{\sigma}/\eps}
   \frac{|y'|^3}{(1+|y'|^2)^{N-1}}  
   \,dy'\right)\\
   &&+\ O(\eps^{N-1})\\
   &=&O(\eps^2).
\eas

Combining the estimates for the three terms on the right hand side 
of (\ref{sumudt}),
\be\label{quaqua}
|U_{\eps}|_{\tz}^\tz=\frac{\tz}{2}B(N)\eps+O(\eps^2),
\ee
if all the $\lambda_{i}$'s are positive.  If all the $\lambda_{i}$'s 
are negative, then the minus sign on the right hand side of (\ref{sumudt})
turns into a plus sign, and (\ref{quaqua}) follows.  From these two 
cases we deduce that (\ref{quaqua}) holds no matter what the 
sign of the $\lambda_{i}$'s is.

\vspace{4mm}

\noindent {\bf Acknowledgments.}
This work was initiated during the visit of the first 
author to Instituto Superior T\'{e}cnico, 
whose kind hospitality and support he gratefully acknowledges.

\end{document}